\documentclass[a4paper,12pt]{amsart}
\usepackage[utf8]{inputenc}
\usepackage{amsmath,amssymb,amsthm}
\usepackage{a4wide}

\theoremstyle{plain}
\newtheorem{thm}{Theorem}
\newtheorem{prop}[thm]{Proposition}
\newtheorem{lemma}[thm]{Lemma}
\newtheorem{cor}[thm]{Corollary}

\theoremstyle{definition}
\newtheorem{defn}[thm]{Definition}
\newtheorem{rmk}[thm]{Remark}

\begin{document}


\title{Exotic one-parameter semigroups of endomorphisms of a symmetric cone}
\author{Bojan Kuzma, Matja\v{z} Omladi\v{c}, Klemen \v{S}ivic, Josef Teichmann}
\address{University of Primorska, Faculty of mathematics, natural sciences and information technologies, Glagolja\v{s}ka 8, SI-6000 Koper, Slovenia (B. Kuzma)\\ University of Ljubljana, Faculty of mathematics and physics, Jadranska ulica 19, SI-1000 Ljubljana, Slovenia (M. Omladi\v{c}, K. \v{S}ivic)\\ ETH Z\"urich, D-Math, R\"amistrasse 101, CH-8092 Z\"urich, Switzerland (K. \v{S}ivic, J. Teichmann)}
\email{bojan.kuzma@upr.si, matjaz.omladic@fmf.uni-lj.si, klemen.sivic@fmf.uni-lj.si, jteichma@math.ethz.ch}
\thanks{The second, third and fourth author gratefully acknowledge the support from
ETH-foundation and by Sciex.}
\curraddr{}
\begin{abstract}
We construct an exotic one-parameter semigroup of endomorphims of a symmetric cone $C$, whose generator is not the sum of a Lie group generator and an endomorphism of $C$. The question is motivated by the theory of affine processes on symmetric cones, which plays an important role in mathematical finance. On the other hand, theoretical question that we solve in this paper
seems to have been implicitly open even much longer then this
motivation suggests.

\end{abstract}


\date{\today}

\maketitle

\section{Introduction}\label{sec:introduction}

Let us start by our main motivation and postpone the firm theoretical background of our investigation until a later point in this introduction.
Affine processes play a major role in mathematical finance, since they are quite flexible from a modeling point of view on the one hand, and very tractable from the point of view of numerics and calibration on the other hand. Both aspects are equally important, when it comes to applications in financial industry. Multi-variate extensions of well-known classical models have received a lot of attention recently, for instance affine processes on the so called canonical state space $ \mathbb{R}^m_{\geq 0} \times \mathbb{R}^n $, or affine processes on positive semidefinite matrices. Both cases are related to affine processes on symmetric cones, where a general theory has been developed in \cite{CK-RMT}.

In \cite{CK-RMT} a conjecture on drifts of affine processes has been stated: let $C \subseteq V $ denote a symmetric cone and let $ B : V \to V $ be a linear map such that
\[
\langle B(u), v \rangle \geq 0 
\]
for all $ u,v \in C $ satisfying $ \langle u , v \rangle = 0 $. Such operators are just the generators of one-parameter semigroups of linear endomorphims of $C$. These generators are also referred to as inward-pointing maps. Particular examples of such maps are Lie algebra elements, i.e.~generators of one-parameter groups of linear automorphisms of $C$, or simply endomorphisms of $C$. The conjecture states that any inward-pointing drift $B$ can be written as a sum of a Lie algebra element and an endomorphism of $C$.

There is in fact a second, somehow easier but related question: whether an inward-pointing drift can be written as a sum of a Lie algebra element and a completely positive drift. Notice, however, that many classical counterexamples of positive and not completely positive maps, such as Choi's map \cite{C2}, can in fact be written as sum of Lie algebra elements and completely positive maps, hence the question cannot be answered immediately. One interpretation of this second conjecture is the following: assume
\[
B(u) = L(u) + \int_C \langle x , u \rangle \mu(dx) 
\]
where $L$ is a generator of one-parameter group of linear automorphisms of $C$, and assume $B$ to be the drift of an affine process with self-exciting jumps described by $\mu$. The second term on the right hand side is just another way to write a completely positive map. Then we can find an equivalent measure change such that the drift $B$ minus the compensator of linear jumps disappears.

In the present paper we construct counterexamples to the above conjecture on all irreducible symmetric cones of rank greater than $2$. We refer to those drifts as {\em exotic drifts}, we also call the generated semigroup {\em exotic semigroup}. Exotic drifts cannot be removed by equivalent measure changes. A surprising and related statement can also be formulated: there is no single non-vanishing jump structure with non-scalar compensator for an affine process such that the exotic drift $B$ constructed in our paper satisfies the admissibility condition
\[
\langle B(u), v \rangle - \int_C \langle x , u \rangle \langle v, \mu(d x) \rangle \geq 0
\]
for all $ u,v \in C $ satisfying $ \langle u , v \rangle = 0 $. In other words these exotic drifts do never appear as drifts of affine self-exciting jump diffusions with compensated excited jump structure as long as the compensator is non-scalar. In the language of linear algebra this means that there exists no non-scalar (completely) positive map $B'$ such that $B-B'$ is a generator of one-parameter semigroup of positive maps. We have a proof of this result if $C$ is the cone of all real positive semidefinite $3\times 3$ matrices and we conjecture that it holds in any symmetric cone.

We are now ready to give the correct linear algebra background of this paper.
Let $V$ be a finite dimensional real vector space equipped with a scalar product $\langle \cdot ,\cdot \rangle$, and let $C\subseteq V$ be some pointed convex cone (i.e. it contains 0). We will assume that the cone $C$ is {\em proper}, i.e. $C\cap (-C)=\{0\}$, and {\em generating}, i.e. $C-C=V$. The cone $C$ induces a partial order on $V$, which will be denoted by $\le$, i.e. $v\le u$ if and only if $u-v\in C$.

The cone
$$C^*=\{v\in V;\langle u,v\rangle \ge 0\, \mathrm{for}\, \mathrm{all}\, u\in C\}$$
is called the {\em closed dual cone} of $C$. Similarly, if the cone $C$ is open, then the cone
$$\{v\in V;\langle u,v\rangle >0\, \mathrm{for}\, \mathrm{all}\, u\in \overline{C}\backslash \{0\}\}$$
is called the {\em open dual cone} of $C$. The open dual cone of an open cone $C$ is the interior of the closed dual cone of $\overline{C}$ (see \cite[Chapter I]{FK}), therefore we will denote it by $\mathrm{int}(\overline{C}^*)$. An open (respectively closed) cone $C$ is called {\em self-dual} if $C=\mathrm{int}(\overline{C}^*)$ (respectively $C=C^*$).

{\em Endomorphism} of the cone $C$ is a linear map $A\colon V\to V$ satisfying $A(C)\subseteq C$. As usual, the set of all endomorphisms of $C$ will be denoted by $\mathrm{End}(C)$. Clearly this is a convex cone in $\mathrm{End}(V)$. Note that endomorphisms of the cone $C$ are often called also {\em $C$-positive maps} or {\em positive maps}. If an element $A\in \mathrm{End}(C)$ is bijective and it satisfies $A(C)=C$, then it is called an {\em automorphism} of $C$. The set of all automorphisms of the cone $C$ is a group for multiplication which is called the {\em automorphism group} of $C$ and denoted by $G(C)$. 
The cone $C$ is called {\em homogeneous} if $G(C)$ acts transitively on $C$, i.e. if for all $u,v\in C$ there exists $g\in G(C)$ such that $gu=v$. Clearly, only open cones can be homogeneous. An open cone is called {\em symmetric} if it is homogeneous and self-dual. By abuse of notation we will call a closed cone {\em symmetric} if it is self-dual and its interior is homogeneous. A symmetric cone $C$ is {\em reducible} if there exist proper vector subspaces $V_1,V_2\subseteq V$ containing symmetric cones $C_1\subseteq V_1$ and $C_2\subseteq V_2$ such that $V$ is a direct sum of $V_1$ and $V_2$ and $C=C_1+C_2$. Otherwise $C$ is {\em irreducible}.

It is well-known that a closed convex cone $C$ in the Euclidean space $V$ is symmetric if and only if $V$ admits a Jordan product such that $V$ is unital Euclidean Jordan algebra and $C=\{x^2;x\in V\}$, see \cite[Chapter III]{FK}. Moreover, the cone $C$ is irreducible if and only if the Jordan algebra $V$ is simple. Furthermore, simple Euclidean Jordan algebras are classified, see \cite[Chapter V]{FK}. They are of the following 2 types:
\begin{enumerate}
\item
$V$ is equal to the real vector space $H_n(\mathbb{D})$ of all hermitian $n\times n$ matrices over real numbers, complex numbers, quaternions or octonions (i.e. $\mathbb{D}\in \{\mathbb{R},\mathbb{C},\mathbb{H},\mathbb{O}\}$). Moreover, if $\mathbb{D}=\mathbb{O}$, then $n\le 3$. The algebra $\mathbb{D}$ is called {\em Hurwitz algebra} or {\em composition algebra}. The symmetric cone $C=H_n(\mathbb{D})^+$ is the cone of all positive semidefinite $n\times n$ matrices over $\mathbb{D}$, the Jordan product is the usual Jordan product $X\circ Y=\frac{1}{2}(XY+YX)$, and $\langle X,Y\rangle =\mathrm{Re}(\mathrm{Tr}(X\circ Y))$.
\item
$V=\mathbb{R}\times \mathbb{R}^{n-1}$ where $\langle \cdot ,\cdot \rangle$ is the usual scalar product on $\mathbb{R}^n$ and the Jordan product is defined by $(\lambda ,u)\circ (\mu ,v)=(\lambda \mu +\langle u,v\rangle ,\lambda v+\mu v)$. In this case $C$ is the {\em Lorentz cone}
$$\Lambda _n=\{(\lambda ,u)\in \mathbb{R}\times \mathbb{R}^{n-1};\lambda \ge 0,\lambda ^2\ge ||u||^2\}.$$
\end{enumerate}
The {\em rank} of a Euclidean Jordan algebra $V$ is the largest possible degree of minimal polynomial of an element of $V$. For the symmetric cone $C=\{x^2;x\in V\}$ we define $\mathrm{rank}\, C=\mathrm{rank}\, V$. The classification of Euclidean Jordan algebras in Chapter V of \cite{FK} shows that the rank of each Lorentz cone is 2, while $\mathrm{rank}\, H_n(\mathbb{D})^+=n$ for each $n$, where $\mathbb{D}\in \{\mathbb{R},\mathbb{C},\mathbb{H},\mathbb{O}\}$ (and $\mathbb{D}=\mathbb{O}$ only if $n\le 3$). Moreover, by Corollary IV.1.5 of \cite{FK} each irreducible symmetric cone of rank 2 is isomorphic to some Lorentz cone. In particular, $H_2(\mathbb{R})^+$, $H_2(\mathbb{C})^+$, $H_2(\mathbb{H})^+$ and $H_2(\mathbb{O})^+$ are isomorphic to some Lorentz cones.

In what follows we will assume throughout that the cone $C$ is closed. The {\em Lie algebra} $\mathfrak{g}(C)$ of the automorphism group $G(C)$ of $C$ consists of all generators of 1-parameter groups of automorphisms of $C$. i.e.
$$\mathfrak{g}(C)=\{A\in \mathrm{End}(V);e^{tA}\in G(C)\, \mathrm{for}\, \mathrm{all}\, t\in \mathbb{R}\}$$
where $e^{tA}=\sum _{n=0}^{\infty}\frac{t^nA^n}{n!}$ is the usual exponential function of an operator. By Lemma 5.1 of \cite{GKT} the Lie algebra $\mathfrak{g}(C)$ can be equivalently described as the space of all maps $A\in \mathrm{End}(V)$ such that $e^{tA}$ maps the boundary of $C$ to itself for each $t\ge 0$. Similarly we define
$$\mathfrak{s}(C)=\{A\in \mathrm{End}(V);e^{tA}\in \mathrm{End}(C)\, \mathrm{for}\, \mathrm{all}\, t\ge 0\},$$
i.e. $\mathfrak{s}(C)$ consists of all generators of 1-parameter semigroups of $C$-positive maps. The notations $\mathfrak{g}(C)$ and $\mathfrak{s}(C)$ indicate that the elements of these sets generate 1-parameter groups respectively semigroups of $C$-positive maps. In the literature the elements of $\mathfrak{s}(C)$ are usually called {\em cross-positive maps} or {\em exponentially positive maps}. Clearly $\mathfrak{g}(C)=\mathfrak{s}(C)\cap (-\mathfrak{s}(C))$. For the properties of $\mathfrak{s}(C)$ see \cite{EH-O}, \cite{AN} and Chapter 4 of \cite{EK}. For our purpose the following characterization, which was proved by Schneider and Vidyasagar \cite[Theorem 3]{SV}, will be very useful.

\begin{thm}[Schneider, Vidyasagar]\label{orthogonal->nonnegative}
Let $A\colon V\to V$ be a linear map. Then $A\in \mathfrak{s}(C)$ if and only if $\langle Au,v\rangle \ge 0$ for all $u\in C$ and $v\in C^*$ that satisfy $\langle u,v\rangle =0$.
\end{thm}

Apply this theorem to both $A$ and $-A$ in order to see:

\begin{cor}\label{orthogonal->0}
A linear map $A\colon V\to V$ belongs to $\mathfrak{g}(C)$ if and only if $\langle Au,v\rangle =0$ for all $u\in C$ and $v\in C^*$ that satisfy $\langle u,v\rangle =0$.
\end{cor}

Theorem \ref{orthogonal->nonnegative} implies that $\mathfrak{s}(C)$ is a closed convex cone in $\mathrm{End}(V)$ which contains the cone $\mathrm{End}(C)+\mathfrak{g}(C)$. However, it is not clear whether $\mathfrak{s}(C)=\mathrm{End}(C)+\mathfrak{g}(C)$. Lemma 6 and Theorem 2 of \cite{SV} imply that $\mathrm{End}(C)+\mathbb{R}\cdot \mathrm{id}_V$ is dense in $\mathfrak{s}(C)$, and since $\mathrm{id}_V\in \mathfrak{g}(C)$, we obtain $\mathfrak{s}(C)=\overline{\mathrm{End}(C)+\mathfrak{g}(C)}$. Therefore the equality $\mathfrak{s}(C)=\mathrm{End}(C)+\mathfrak{g}(C)$ holds if and only if the cone $\mathrm{End}(C)+\mathfrak{g}(C)$ is closed. However, this depends on the geometry of the cone $C$, and usually it is not true. In particular, Gritzmann, Klee and Tam \cite{GKT} proved that $\mathrm{End}(C)+\mathfrak{g}(C)\ne \mathfrak{s}(C)$ if $\dim V\ge 3$, $\mathfrak{g}(C)=\mathbb{R}\cdot \mathrm{id}_V$ and the cone $C$ is strictly convex or smooth. On the other hand, Schneider and Vidyasagar \cite{SV} showed that the equality $\mathrm{End}(C)+\mathbb{R}\cdot \mathrm{id}_V=\mathfrak{s}(C)$ holds if $C$ is a polyhedral cone.

The aim of this paper is to characterize symmetric cones $C$ for which the equality $\mathrm{End}(C)+\mathfrak{g}(C)=\mathfrak{s}(C)$ holds. From our overview given here it will have become clear to the reader that this question emerges from the work in the above mentioned area.
Note that Gritzmann's, Klee's and Tam's result does not apply to symmetric cones, because they have big automorphism groups. As we pointed out above, our interest in 1-parameter semigroups of positive maps on symmetric cones comes from the study of afine processes. In \cite{CK-RMT} the affine processes on symmetric cones were characterized (see also \cite{CFMT} for the characterization in the case of real positive semidefinite matrices) and the drift part of an affine process corresponds to 1-parameter semigroup of positive maps on $V$ (see \cite[Theorem 2.4]{CK-RMT}). In particular, in Section 2.1.2 of \cite{CFMT} it was conjectured that $\mathfrak{s}(H_n(\mathbb{R})^+)=\mathrm{End}(H_n(\mathbb{R})^+)+\mathfrak{g}(H_n(\mathbb{R})^+)$. The same conjecture was stated in Section 4 of \cite{D}.

The analogous result for completely positive maps was proved in \cite{L}. A linear map $\phi \colon \mathcal{A}\to \mathcal{B}$ between (real or complex) $C^*$-algebras $\mathcal{A}$ and $\mathcal{B}$ is called {\em $n$-positive} if the linear map $\phi _n\colon M_n(\mathcal{A})\to M_n(\mathcal{B})$ defined by $\phi _n\big([a_{ij}]_{i,j=1}^n\big)=[\phi (a_{ij})]_{i,j=1}^n$ is positive. A linear map $\phi$ between $C^*$-algebras is called {\em completely positive} if it is $n$-positive for each positive integer $n$. In our motivation at the beginning of this introduction a slightly different definition of complete positivity was actually discussed, a definition that is used on symmetric cones: a map $ \phi: C \to C $ is called {\em completely positive}, if there exists a $V$-valued finite measure $ \mu $ on the Borel sets of $V$, such that $ \langle v , \mu \rangle $ is a positive measure supported on $C$ for all $ v \in C $, and if
\[
\phi(u) = \int_C \langle x, u \rangle \mu (d x)
\]
holds for all $ u \in C$. This is equivalent to saying that a map is completely positive if it lies in the closed cone hull of maps of the type $ u \mapsto \langle x_1, u \rangle x_2 $ for some $ x_1, x_2 \in C $. Notice that a completely positive map can be written as a finite sum of such maps by Tchakaloff's theorem \cite{BT}. Therefore the definition of completely positive maps on symmetric cones coincides with the previous definition on $C^*$-algebras by \cite{C}. It follows that from the point of view of C*-algebras, the question that we are solving in this paper seems to be implicitly open for an even longer period then from the point of view of mathematical finance.

The (real) finite-dimensional version of Theorem 3 of \cite{L} implies that each linear map on $H_n(\mathbb{R})$ that is a generator of 1-parameter semigroup of completely positive maps can be written as a sum of a completely positive map and an element of the Lie algebra $\mathfrak{g}(H_n(\mathbb{R})^+)$. However, no results are known about the structure of generators of 1-parameter semigroups of positive, but not completely positive maps on $H_n(\mathbb{R})$. Much more is known for the Lorentz cones $\Lambda _n$. In this case positive maps were completely characterized in \cite{LS}, and cross-positive maps were completely characterized in \cite{SW1}. Using these results Stern and Wolkowicz \cite[Theorem 4.2]{SW2} proved the following theorem.

\begin{thm}[Stern, Wolkowicz]\label{Lorentz}
$\mathfrak{s}(\Lambda _n)=\mathrm{End}(\Lambda _n)+\mathfrak{g}(\Lambda _n)$ for each positive integer $n$.
\end{thm}

Since $\mathbb{R}^+$ is the only symmetric cone of rank 1 and each irreducible symmetric cone of rank 2 is isomorphic to some Lorentz cone, Theorem \ref{Lorentz} yields:

\begin{cor}
Let $V$ be a simple Euclidean Jordan algebra of rank at most 2 and $C=\{x^2;x\in V\}$. Then $\mathfrak{s}(C)=\mathrm{End}(C)+\mathfrak{g}(C)$.
\end{cor}

The aim of this paper is to prove the converse of this corollary, i.e. we will show:

\begin{thm}\label{decomposition<->rank2}
Let $V$ be a simple Euclidean Jordan algebra and $C=\{x^2;x\in V\}$ the symmetric cone of all squares of $V$. Then $\mathfrak{s}(C)=\mathrm{End}(C)+\mathfrak{g}(C)$ if and only if $\mathrm{rank}\, V\le 2$.
\end{thm}

More precisely, we will prove:

\begin{thm}\label{B-counterexample}
Let $n\ge 3$ and $\mathbb{D}\in \{\mathbb{R},\mathbb{C},\mathbb{H},\mathbb{O}\}$ with $\mathbb{D}=\mathbb{O}$ only for $n=3$, and let $B\colon H_n(\mathbb{D})\to H_n(\mathbb{D})$ be the linear map defined by
$$B\left(\left[
\begin{array}{ccccc}x_{11}&x_{12}&x_{13}&\cdots&x_{1n}\\ \overline{x_{12}}&x_{22}&x_{23}&\cdots&x_{2n}\\ \overline{x_{13}}&\overline{x_{23}}&x_{33}&\cdots& x_{3n}\\\vdots&\vdots&\vdots&\ddots&\vdots\\ \overline{x_{1n}}&\overline{x_{2n}}&\overline{x_{3n}} &\cdots&x_{nn}
\end{array}
\right]\right) =$$
$$=\left[
\begin{array}{ccccc}x_{22}-x_{11}&-px_{12}&-px_{13}&\cdots&-px_{1n}\\-p\overline{x_{12}}&(n-1)^2(x_{33}-x_{22})&-qx_{23}&\cdots&-qx_{2n}\\-p\overline{x_{13}}&-q\overline{x_{23}}&(n-1)^2(x_{44}-x_{33})&\cdots&-qx_{3n}\\\vdots&\vdots&\vdots&\ddots&\vdots\\ -p\overline{x_{1n}}&-q\overline{x_{2n}}&-q\overline{x_{3n}}&\cdots&(n-1)^2(x_{11}-x_{nn})
\end{array}
\right]$$
where $p=\frac{(n-2)(n^2-n-1)}{2}$ and $q=(n-1)(n^2-n-1)$. Then $B$ is generator of 1-parameter semigroup of positive maps on $H_n(\mathbb{D})$, but it does not belong to the sum $\mathrm{End}(H_n(\mathbb{D})^+)+\mathfrak{g}(H_n(\mathbb{D})^+)$.
\end{thm}

The paper is organized as follows. In Section 2 we prove some basic results on Hurwitz algebras and Euclidean Jordan algebras, that will be needed in the sequel. We show that Theorem \ref{orthogonal->nonnegative} and Corollary \ref{orthogonal->0} need to be verified only if $u$ and $v$ are primitive idempotents in the Jordan algebra. Then we characterize the primitive idempotents in $H_n(\mathbb{D})$ where $\mathbb{D}$ is a Hurwitz algebra and $n$ is a positive integer (with $n\le 3$ if $\mathbb{D}=\mathbb{O}$). We also characterize the elements of the Lie algebra $\mathfrak{g}(H_n(\mathbb{D})^+)$. In the rest of the paper we prove Theorem \ref{B-counterexample}. In Section 3 we show that the map $B$ from that theorem is generator of 1-parameter semigroup of positive maps on $H_n(\mathbb{D})$, and in Section 4 we show that it cannot be written as a sum of a positive map and an element of the Lie algebra $\mathfrak{g}(H_n(\mathbb{D})^+)$.

\section{Hurwitz algebras and Euclidean Jordan algebras}

In this section we prove some results on Hurwitz algebras and Euclidean Jordan algebras that will be needed in the proof of Theorem \ref{B-counterexample}. The references for this section are \cite{FK} and \cite{SpV}.

A {\em Hurwitz algebra} or {\em decomposition algebra} $\mathbb{D}$ over $\mathbb{R}$ is a not necessarily associative algebra over $\mathbb{R}$ with identity such that there exists a nondegenerate quadratic form $Q$ on $\mathbb{D}$ satisfying $Q(xy)=Q(x)Q(y)$ for all $x,y\in \mathbb{D}$. A Hurwitz algebra is called {\em Euclidean} if the quadratic form $Q$ is positive definite. If $\mathbb{D}$ is a Euclidean Hurwitz algebra, then the square of the quadratic form $Q$ is called a {\em norm} or an {\em absolute value} and we write $|x|=\sqrt{Q(x)}$ for $x\in \mathbb{D}$. The absolute value defines a {\em scalar product} $\langle \cdot ,\cdot \rangle _{\mathbb{D}}$ on $\mathbb{D}$ by $\langle x,y\rangle _{\mathbb{D}}=\frac{1}{4}(|x+y|^2-|x-y|^2)$. The {\em conjugate} of $x\in \mathbb{D}$ is defined by $\overline{x}=2\langle x,1\rangle _{\mathbb{D}}-x$, and the {\em real part} by $\mathrm{Re}\, x=\langle x,1\rangle _{\mathbb{D}}=\frac{1}{2}(x+\overline{x})$. Note that the real numbers commute with every element of $\mathbb{D}$. In the sequel the following identities will be very useful. They can be found in \cite[Lemmas 1.3.1, 1.3.2, 1.3.3]{SpV}, therefore we state them without the proof.

\begin{lemma}\label{Hurwitz algebra identities}
Let $\mathbb{D}$ be a Euclidean Hurwitz algebra and $x,y,z\in \mathbb{D}$ arbitrary elements. Then the following identities hold.
\begin{enumerate}
\item
$x\overline{x}=\overline{x}x=|x|^2$. In particular, $\mathbb{D}$ is a division algebra with the inverse defined by $x^{-1}=\frac{\overline{x}}{|x|^2}$.
\item
$\overline{xy}=\overline{y}\cdot \overline{x}$.
\item
$\langle \overline{x},\overline{y}\rangle _{\mathbb{D}}=\langle x,y\rangle _{\mathbb{D}}$.
\item
$\langle xy,z\rangle _{\mathbb{D}}=\langle y,\overline{x}z\rangle _{\mathbb{D}}$ and $\langle xy,z\rangle _{\mathbb{D}}=\langle x,z\overline{y}\rangle _{\mathbb{D}}$.
\item
$x(\overline{x}y)=|x|^2y$ and $(x\overline{y})y=|y|^2x$.
\end{enumerate}
\end{lemma}

By Theorem V.1.5 of \cite{FK} the only Euclidean Hurwitz algebras are $\mathbb{R}$, $\mathbb{C}$, $\mathbb{H}$ and $\mathbb{O}$. Let $d=\dim _{\mathbb{R}}\mathbb{D}$, i.e. $d=\left\{
\begin{array}{l}1\, \mathrm{if}\, \mathbb{D}=\mathbb{R}\\2\, \mathrm{if}\, \mathbb{D}=\mathbb{C}\\4\, \mathrm{if}\, \mathbb{D}=\mathbb{H}\\8\, \mathrm{if}\, \mathbb{D}=\mathbb{O}
\end{array}
\right.$, and let $\{f_1,\ldots ,f_d\}$ be the standard basis of $\mathbb{D}$. Then
$$\mathrm{Re}\left(\sum _{m=1}^dx_mf_m\right)=x_1,\quad \overline{\sum _{m=1}^dx_mf_m}=x_1-\sum _{m=2}^dx_mf_m,\quad \left|\sum _{m=1}^dx_mf_m\right| =\sqrt{\sum _{m=1}^dx_m^2}$$
and
$$\langle \sum _{m=1}^dx_mf_m,\sum _{m=1}^dy_mf_m\rangle _{\mathbb{D}}=\sum _{m=1}^dx_my_m =\mathrm{Re}\left(\Big(\sum _{m=1}^dx_mf_m\Big)\Big(\overline{\sum _{m=1}^dy_mf_m}\Big)\right).$$
The previous lemma now implies the following.

\begin{cor}\label{Hurwitz algebra real parts}
Let $\mathbb{D}$ be a Euclidean Hurwitz algebra and $x,y,z\in \mathbb{D}$ arbitrary elements. Then $\mathrm{Re}(xy)=\mathrm{Re}(yx)$ and $\mathrm{Re}(x(yz))=\mathrm{Re}((xy)z)$.
\end{cor}
\begin{proof}
The first identity is clear, and the second one follows from the following chain of identities:
$$\mathrm{Re}((xy)z)=\langle xy,\overline{z}\rangle _{\mathbb{D}}=\langle x,\overline{z}\cdot \overline{y}\rangle _{\mathbb{D}}=\langle x,\overline{yz}\rangle _{\mathbb{D}}=\mathrm{Re}(x(yz)).$$
\end{proof}

\begin{rmk}
Since $\mathrm{Re}(x(yz))=\mathrm{Re}((xy)z)$ for all $x,y,z\in \mathbb{D}$, we write $\mathrm{Re}(xyz)$ instead of $\mathrm{Re}(x(yz))$ and $\mathrm{Re}((xy)z)$.
\end{rmk}

\begin{rmk}
Note that although $\mathrm{Re}((xy)z)=\mathrm{Re}(x(yz))$ for all $x,y,z\in \mathbb{D}$, the equality $\mathrm{Re}\Big((x(yz))w\Big)=\mathrm{Re}\Big(((xy) z)w\Big)$ does not hold for all $x,y,z,w\in \mathbb{D}$, since this would imply $x(yz)=(xy)z$ for all $x,y,z\in \mathbb{D}$ (see the proof of Lemma \ref{zeros of positive map} below) which is not true in the case of octonions. However, we do have the following.
\end{rmk}

\begin{cor}\label{Re(x^-1xyz)}
If $\mathbb{D}$ is a Euclidean Hurwitz algebra and $x,y,z\in \mathbb{D}$ arbitrary elements, then $\mathrm{Re}\Big(\overline{x}((xy)z)\Big)=|x|^2\mathrm{Re}(yz)$.
\end{cor}
\begin{proof}
By Corollary \ref{Hurwitz algebra real parts} and Lemma \ref{Hurwitz algebra identities} we have
$$\mathrm{Re}\Big(\overline{x}((xy)z)\Big)=\mathrm{Re}\Big((\overline{x}(xy))z\Big)=\mathrm{Re}(|x|^2yz).$$
\end{proof}

By the classification of Euclidean Jordan algebras \cite[Chapter V]{FK} each simple Euclidean Jordan algebra of rank at least 3 is isomorphic to some Jordan algebra $H_n(\mathbb{D})$ of all hermitian $n\times n$ matrices (with respect to the usual involution $X^*=\overline{X^T}=(\overline{X})^T$) over a Hurwitz algebra $\mathbb{D}$ (where $\mathbb{D}=\mathbb{O}$ only if $n=3$), the Jordan product is defined by $X\circ Y=\frac{1}{2}(XY+YX)$, and the scalar product by $\langle X,Y\rangle =\mathrm{Re}\, \mathrm{Tr}(X\circ Y)$. However, since $H_n(\mathbb{D})$ is a Jordan algebra, the matrix $X\circ Y$ belongs to $H_n(\mathbb{D})$ for all $X,Y\in H_n(\mathbb{D})$. In particular, $X\circ Y$ has a real diagonal, therefore $\langle X,Y\rangle =\mathrm{Tr}(X\circ Y)$ for all $X,Y\in H_n(\mathbb{D})$.

Now we will characterize the elements of $\mathfrak{g}(H_n(\mathbb{D})^+)$. This characterization will be useful in Section 4, where we will prove that the linear map $B$ from Theorem \ref{B-counterexample} does not belong to $\mathrm{End}(H_n(\mathbb{D})^+)+\mathfrak{g}(H_n(\mathbb{D})^+)$. Note that the elements of $\mathfrak{g}(H_n(\mathbb{R})^+)$ and $\mathfrak{g}(H_n(\mathbb{C})^+)$ were characterized already in \cite{D}, but our proof is shorter and it characterizes the elements of $\mathfrak{g}(H_n(\mathbb{D})^+)$ also for $\mathbb{D}=\mathbb{H}$ and $\mathbb{D}=\mathbb{O}$. We will need the following definition.

\begin{defn}\label{quadratic_representation}
Let $V$ be a Euclidean Jordan algebra and $x\in V$ arbitrary element. The {\em multiplication operator} $L(x)\colon V\to V$ is defined by $L(x)y=x\circ y$. We define also $P(x)=2L(x)^2-L(x^2)$. The map $P$ is called the {\em quadratic representation} of $V$.
\end{defn}

In particular, if $V=H_n(\mathbb{D})$ where $\mathbb{D}$ is a Euclidean Hurwitz algebra (with $\mathbb{D}=\mathbb{O}$ only for $n\le 3$), then $L(X)Y=\frac{1}{2}(XY+YX)$ and
$$P(X)Y=\frac{1}{2}\Big(X(YX)+(XY)X+X(XY)+(YX)X-X^2Y-YX^2\Big).$$
In particular, if $\mathbb{D}$ is associative, then $P(X)Y=XYX$.

\begin{prop}\label{Lie algebra}
Let $V$ be a simple Euclidean Jordan algebra with $C=\{x^2;x\in V\}$, and let $A\colon V\to V$ be a linear map. Then $A\in \mathfrak{g}(C)$ if and only if there exist elements $y_1,\ldots ,y_m,z_1,\ldots ,z_m,w\in V$ such that $A=\sum _{l=1}^m[L(y_l),L(z_l)]+L(w)$, i.e.
$$Ax=\sum _{l=1}^m(y_l\circ (z_l\circ x)-z_l\circ (y_l\circ x))+w\circ x$$
for each $x\in V$.

In particular, $A\in \mathfrak{g}(H_n(\mathbb{D})^+)$, where $\mathbb{D}$ is a Euclidean Hurwitz algebra (with $\mathbb{D}=\mathbb{O}$ only for $n\le 3$), if and only if there exist $Y_1,\ldots ,Y_m,Z_1,\ldots ,Z_l,W\in H_n(\mathbb{D})$ such that
$$A(X)=\frac{1}{2}(WX+XW)\quad +$$
$$+\frac{1}{4}\sum _{l=1}^m\Big(Y_l(Z_lX)+Y_l(XZ_l)+(Z_lX)Y_l+(XZ_l)Y_l-Z_l(Y_lX)-Z_l(XY_l)-(Y_lX)Z_l-(XY_l)Z_l\Big)$$
for each $X\in H_n(\mathbb{D})$. Moreover, if $\mathbb{D}\ne \mathbb{O}$, then the above condition is equivalent to
$$A(X)=\frac{1}{4}\sum _{l=1}^m[[Y_l,Z_l],X\big]+\frac{1}{2}(WX+XW)$$
for each $X\in H_n(\mathbb{D})$.
\end{prop}
\begin{proof}
We will use the notation from \cite{FK}. As explained on the page 6 of \cite{FK}, the Lie algebra $\mathfrak{g}=\mathfrak{g}(C)$ can be decomposed as a vector space direct sum $\mathfrak{g}=\mathfrak{k}+\mathfrak{p}$ where $\mathfrak{k}=\{A\in \mathfrak{g};A^*=-A\}$ and $\mathfrak{p}=\{A\in \mathfrak{g};A^*=A\}$. Here $\mathfrak{k}$ is a subalgebra of $\mathfrak{g}$, $[\mathfrak{p},\mathfrak{p}]\subseteq \mathfrak{k}$ and $[\mathfrak{k},\mathfrak{p}]\subseteq \mathfrak{p}$. Moreover, as explained before Theorem III.3.1 of \cite{FK}, the map $L$ is bijection between $V$ and $\mathfrak{p}$, i.e. $\mathfrak{p}=\{L(w);w\in V\}$. On the other hand, by \cite[Theorem III.5.1]{FK} $\mathfrak{k}=\mathrm{Der}(V)$ is the derivation algebra of $V$. However, since $V$ is simple, by Theorem 2 of \cite{J} each derivation of $V$ is inner, i.e. of the form $\sum _{l=1}^m[L(y_l),L(z_l)]$ for some $y_1,\ldots ,y_m,z_1,\ldots ,z_m\in V$. This proves the first part of the proposition, and all the rest is just computation, in the last part using associativity of $H_n(\mathbb{D})$ if $\mathbb{D}\in \{\mathbb{R},\mathbb{C},\mathbb{H}\}$.
\end{proof}

\begin{cor}\label{Lie algebraRCH}
Let $\mathbb{D}\in \{\mathbb{R},\mathbb{C},\mathbb{H}\}$ and let $A\colon H_n(\mathbb{D})\to H_n(\mathbb{D})$ be a linear map. Then $A\in \mathfrak{g}(H_n(\mathbb{D})^+)$ if and only if there exists $H\in M_n(\mathbb{D})$ such that $A(X)=HX+XH^*$ for each $X\in H_n(\mathbb{D})$.
\end{cor}
\begin{proof}
We already know that $\mathfrak{g}(H_n(\mathbb{D})^+)=\mathrm{Der}(H_n(\mathbb{D}))+\mathfrak{p}$ where $\mathfrak{p}$ consists of all linear maps of the form $X\mapsto XW+WX$ where $W\in H_n(\mathbb{D})$. Therefore it suffices to show that $\mathrm{Der}(H_n(\mathbb{D}))$ consists of all linear maps of the form $X\mapsto W'X+XW'^*=[W',X]$ where $W'\in M_n(\mathbb{D})$ is skew-hermitian. Since
$$[W',X\circ Y]=\frac{1}{2}(W'XY+W'YX-XYW'-YXW')=[W',X]\circ Y+X\circ [W',Y],$$
all these maps are indeed derivations. On the other hand, there are no other derivations, since we already know that $\mathrm{Der}(V)$ consists of all linear maps of the form $X\mapsto \sum _{l=1}^m\big[[Y_l,Z_l],X\big]$ where $Y_1,\ldots ,Y_m,Z_1,\ldots ,Z_m\in H_n(\mathbb{D})$, and for each $l=1,\ldots ,m$ the matrix $[Y_l,Z_l]$ is skew-hermitian.
\end{proof}

The characterization of the elements of $\mathfrak{g}(H_3(\mathbb{O})^+)$ is slightly different. The following result is probably well-known.

\begin{lemma}\label{derivation octonions->H_3(O)}
For any derivation $D$ on the octonion algebra $\mathbb{O}$ the linear map $A_D\colon H_3(\mathbb{O})\to H_3(\mathbb{O})$ defined by $A_D([x_{lm}]_{l,m=1}^3)=[D(x_{lm})]_{l,m=1}^3$ is a derivation on the Jordan algebra $H_3(\mathbb{O})$.
\end{lemma}
\begin{proof}
First we note that $D(1)=0$, therefore linearity of $D$ implies that $D$ annihilates all real numbers, and then $D(\overline{x})=-D(x)$ for each $x\in \mathbb{O}$. Moreover, $0=D(|x|^2)=xD(\overline{x})+D(x)\overline{x}=-xD(x)+D(x)\overline{x}$, therefore $xD(x)=D(x)\overline{x}$ for each $x\in \mathbb{O}$. If $x$ is an element of the standard basis of $\mathbb{O}$ which is different from 1, then $\overline{x}=-x$, and the equality $xD(x)=-D(x)x$ implies $\mathrm{Re}(D(x))x=-\mathrm{Re}(D(x))x$, so $\mathrm{Re}(D(x))=0$. By linearity we get $\mathrm{Re}(D(x))=0$ for each $x\in \mathbb{O}$, and therefore $\overline{D(x)}=-D(x)$ for each $x\in \mathbb{O}$. In particular, this shows that the map $A_D$ indeed maps hermitian matrices to hermitian matrices.

For each $l=1,2,3$ denote by $e_l$ the vector with 1 on the $l$-th component and 0 elsewhere. For each $X,Y\in H_3(\mathbb{O})$ and each $l,l'=1,2,3$ we compute
$$e_l^TA_D(X\circ Y)e_{l'}=\frac{1}{2}e_l^TA_D(XY+YX)e_{l'}=\frac{1}{2}D(e_l^T(XY+YX)e_{l'})$$
$$=\frac{1}{2}D\Big(\sum _{m=1}^3(e_l^TXe_me_m^TYe_{l'}+e_l^TYe_me_m^TXe_{l'}) \Big)$$
$$=\frac{1}{2}\sum _{m=1}^3\Big(D(e_l^TXe_me_m^TYe_{l'})+D(e_l^TYe_m e_m^TXe_{l'})\Big)$$
$$=\frac{1}{2}\sum _{m=1}^3\Big(D(e_l^TXe_m)e_m^TYe_{l'}+e_l^TXe_m D(e_m^TYe_{l'})+D(e_l^TYe_m)e_m^TXe_{l'}+e_l^TYe_m D(e_m^TXe_{l'})\Big)$$
$$=\frac{1}{2}\sum _{m=1}^3\Big(e_l^TA_D(X)e_me_m^TYe_{l'}+e_l^TXe_m e_m^TA_D(Y)e_{l'}+e_l^TA_D(Y)e_me_m^TXe_{l'}+e_l^TY e_me_m^TA_D(X)e_{l'}\Big)$$
$$=\frac{1}{2}e_l^T\Big(A_D(X)Y+XA_D(Y)+A_D(Y)X+YA_D(X)\Big) e_{l'}= e_l^T\Big(A_D(X)\circ Y+X\circ A_D(Y)\Big)e_{l'},$$
i.e. $A_D$ is a derivation.
\end{proof}

\begin{rmk}\label{derivationsD->derivationsH_n(D)}
Note that the previous lemma trivially holds if $\mathbb{O}$ were replaced by $\mathbb{D}\in \{\mathbb{R},\mathbb{C},\mathbb{H}\}$ (and the sizes of the matrices were arbitrary), since in that case all derivations on $\mathbb{D}$ are inner, i.e. of the form $x\mapsto ax-xa$ for some $a\in \mathbb{D}$, and then $A_D(X)=aI\cdot X-X\cdot aI$.
\end{rmk}

\begin{cor}\label{derivationsH_3(O)}
A linear map $A\colon H_3(\mathbb{O})\to H_3(\mathbb{O})$ is a derivation if and only if there exist a derivation $D\colon \mathbb{O}\to \mathbb{O}$ and a skew-hermitian matrix $H\in M_3(\mathbb{O})$ with $\mathrm{Tr}(H)=0$ such that $A(X)=[H,X]+A_D(X)$ for each $X\in H_3(\mathbb{O})$, where $A_D\colon H_3(\mathbb{O})\to H_3(\mathbb{O})$ is the linear map defined in Lemma \ref{derivation octonions->H_3(O)}.
\end{cor}
\begin{proof}
We have already proved that $A_D$ is derivation, and by Proposition V.2.3 of \cite{FK} the map $\mathrm{ad}(H)\colon H_3(\mathbb{O})\to H_3(\mathbb{O})$ defined by $\mathrm{ad}(H)X=[H,X]$ is also derivation. Let $V'$ be the vector space of all skew-hermitian matrices with trace zero in $M_3(\mathbb{O})$, and define a linear map $\varphi \colon V'\oplus \mathrm{Der}(\mathbb{O})\to \mathrm{Der}(H_3(\mathbb{O}))$ by $\varphi (H,D)=\mathrm{ad}(H)+A_D$. We will prove that $\varphi$ is injective. Let $(H,D)\in \ker \varphi$. Then $[H,X]+A_D(X)=0$ for each $X\in H_3(\mathbb{O})$. Note that $A_D(X)=0$ for each {\em real} hermitian matrix $X$, therefore $H$ commutes with each real hermitian matrix. This is possible only if $H$ is scalar matrix, and $\mathrm{Tr}(H)=0$ implies $H=0$. However, then $A_D=0$, which clearly implies $D=0$. Therefore $\varphi$ is injective. In particular, $\dim \mathrm{im}\, \varphi =\dim V'+\dim \mathrm{Der}(\mathbb{O})$. Proposition 2.4.5 of \cite{SpV} implies that $\dim \mathrm{Der}(\mathbb{O})=14$, therefore $\dim \mathrm{im}\, \varphi =52$. On the other hand, by the classification of Euclidean Jordan algebras in Chapter V of \cite{FK} the Lie algebra $\mathrm{Der}(H_3(\mathbb{O}))$ is isomorphic to the exceptional Lie algebra $\mathfrak{f}_4$ (see also \cite{CS} and \cite[Section 7.2]{SpV}), therefore $\dim \mathrm{Der}(H_3(\mathbb{O}))=52$. Hence the map $\varphi$ is bijective and corollary follows.
\end{proof}

Since we know (see the proof of Proposition \ref{Lie algebra}) that $\mathfrak{g}(H_3(\mathbb{O})^+)=\mathrm{Der}(H_3(\mathbb{O}))+\mathfrak{p}$ where $\mathfrak{p}=\{L(w):w\in H_3(\mathbb{O})\}$, we obtain:

\begin{cor}\label{Lie algebraO}
Let $A\colon H_3(\mathbb{O})\to H_3(\mathbb{O})$ be a linear map. Then $A\in \mathfrak{g}(H_3(\mathbb{O})^+)$ if and only if there exist a derivation $D\colon \mathbb{O}\to \mathbb{O}$ and $H\in M_3(\mathbb{O})$ with $\mathrm{Tr}(H)\in \mathbb{R}$ such that $A(X)=HX+XH^*+A_D(X)$ for each $X\in H_3(\mathbb{O})$, where $A_D\colon H_3(\mathbb{O})\to H_3(\mathbb{O})$ is the derivation defined in Lemma \ref{derivation octonions->H_3(O)}.
\end{cor}

Let $V$ be an arbitrary Euclidean Jordan algebra and $C=\{v^2;v\in V\}$. An element $c\in V$ is called {\em idempotent} if $c^2=c$. Clearly all idempotents of $V$ belong to the cone $C$. A nonzero idempotent is called {\em primitive} if it cannot be written as a sum of two nonzero idempotents.

We will now show that in the case of symmetric cones the condition $\langle Au,v\rangle \ge 0$ from Theorem \ref{orthogonal->nonnegative} needs to be verified only for primitive idempotents $u,v\in C$ that satisfy $\langle u,v\rangle =0$. If $C$ is a symmetric cone, then similarly to Theorem \ref{orthogonal->nonnegative} self-duality of $C$ implies that a linear map $A$ belongs to $\mathrm{End}(C)$ if and only if $\langle Au,v\rangle \ge 0$ for all $u,v\in C$. We will show that this condition also needs to be verified only for primitive idempotents $u$ and $v$. This will be helpful in the proof of Theorem \ref{B-counterexample}.

\begin{lemma}\label{primitive idempotents}
Let $V$ be a Euclidean Jordan algebra, $C=\{v^2;v\in V\}$, and let $A\colon V\to V$ be a linear map. Then:
\begin{enumerate}
\item[(a)]
$A\in \mathrm{End}(C)$ if and only if $\langle Ax,y\rangle \ge 0$ for all primitive idempotents $x,y\in C$.
\item[(b)]
$A\in \mathfrak{s}(C)$ if and only if $\langle Ax,y\rangle \ge 0$ for all primitive idempotents $x,y\in C$ that satisfy $\langle x,y\rangle =0$.
\item[(c)]
$A\in \mathfrak{g}(C)$ if and only if $\langle Ax,y\rangle =0$ for all primitive idempotents $x,y\in C$ that satisfy $\langle x,y\rangle =0$.
\end{enumerate}
\end{lemma}
\begin{proof}
\begin{enumerate}
\item[(a)]
If $A\in \mathrm{End}(C)$, then clearly $Ax\in C$ for each primitive idempotent $x\in C$, and the self-duality of $C$ implies that $\langle Ax,y\rangle \ge 0$ for all primitive idempotents $x,y\in C$. Conversely, suppose that $\langle Ax,y\rangle \ge 0$ holds for all primitive idempotents $x,y\in C$, and let $u,v\in C$ be arbitrary. By the spectral theorem \cite[Theorem III.1.2]{FK} there exist a Jordan frame (i.e. a complete system of orthogonal primitive idempotents) $x_1,\ldots ,x_n$ and real numbers $\lambda _1,\ldots ,\lambda _n$ such that $u=\sum _{l=1}^n\lambda _lx_l$. Since $u\in C$ and $x_l\in C$ for each $l$, the self-duality of $C$ implies that $\lambda _l=\lambda _l\, \mathrm{Tr}(x_l)=\lambda _l\langle x_l,x_l\rangle =\langle u,x_l\rangle \ge 0$ for each $l$. Similarly, there exist a Jordan frame $y_1,\ldots ,y_n$ and nonnegative numbers $\mu _1,\ldots ,\mu _n$ such that $v=\sum _{l=1}^n\mu _ly_l$. Then
$$\langle Au,v\rangle =\sum _{l,m=1}^n\lambda _l\mu _m\langle Ax_l,y_m\rangle \ge 0,$$
and self-duality of $C$ implies that $Au\in C$. Since $u\in C$ was arbitrary, it follows that $A\in \mathrm{End}(C)$.
\item[(b)]
If $A\in \mathfrak{s}(C)$, then for all primitive idempotents $x,y\in C$ that satisfy $\langle x,y\rangle =0$ the inequality $\langle Ax,y\rangle \ge 0$ follows directly from Theorem \ref{orthogonal->nonnegative} and self-duality of $C$. Conversely, suppose that $\langle Ax,y\rangle \ge 0$ for all primitive idempotents $x,y\in C$ with $\langle x,y\rangle =0$. Let $u,v\in C$ be arbitrary with $\langle u,v\rangle =0$. By the spectral decomposition there exist $\lambda _1,\ldots ,\lambda _n,\mu _1,\ldots ,\mu _n\ge 0$ and Jordan frames $x_1,\ldots ,x_n$ and $y_1,\ldots ,y_n$ such that $u=\sum _{l=1}^n\lambda _lx_l$ and $v=\sum _{l=1}^n\mu _ly_l$. For all $l,m=1,\ldots ,n$ we have $x_l,y_m\in C$, therefore self-duality of $C$ implies that $\langle x_l,y_m\rangle \ge 0$. Since
$$0=\langle u,v\rangle =\sum _{l,m=1}^n\lambda _l\mu _m\langle x_l,y_m\rangle$$
and $\lambda _l,\mu _m,\langle x_l,y_m\rangle \ge 0$ for all $l,m=1,\ldots ,n$, it follows that $\lambda _l\mu _m\langle x_l,y_m\rangle =0$ for all $l,m=1,\ldots ,n$. For each $l,m=1,\ldots ,n$ therefore we have either $\lambda _l\mu _m=0$ or $\langle x_l,y_m\rangle =0$, and in the later case the assumption of the lemma implies that $\langle Ax_l,y_m\rangle \ge 0$. Therefore
$$\langle Au,v\rangle =\sum _{l,m=1}^n\lambda _l\mu _m\langle Ax_l,y_m\rangle \ge 0,$$
i.e. $A\in \mathfrak{s}(C)$.
\item[(c)] follows immediately from (b).
\end{enumerate}
\end{proof}

To use the previous lemma we have to characterize primitive idempotents in $H_n(\mathbb{D})$ where $\mathbb{D}$ is a Euclidean Jordan algebra (and $n=3$ if $\mathbb{D}=\mathbb{O}$).
The following lemma is well-known in the real and complex case.

\begin{lemma}\label{primitive idempotents uu^*}
Let $n$ be a positive integer and $\mathbb{D}\in \{\mathbb{R},\mathbb{C},\mathbb{H},\mathbb{O}\}$, where $\mathbb{D}=\mathbb{O}$ only if $n=3$. Then a matrix $X\in H_n(\mathbb{D})$ is primitive idempotent if and only if $X=uu^*$ for some $u\in \mathbb{D}^n$ with $||u||=1$ and {\em real first component}.
\end{lemma}
\begin{proof}
For $l=1,\ldots ,n$ let $e_l$ denote the vector with 1 on the $l$-th component and 0 elsewhere. The matrices $e_le_l^T=e_le_l^*$ are clearly idempotents, and they are primitive by \cite[Theorem III.1.2]{FK}, since $\mathrm{Tr}(e_le_l^T)=1$.

Let $X\in H_n(\mathbb{D})$ be arbitrary primitive idempotent. By \cite[Corollary IV.2.4(ii)]{FK} there exists $W\in H_n(\mathbb{D})$ with $W^2=I$ such that $P(W)e_1e_1^*=X$, where $P$ is the quadratic representation of $H_n(\mathbb{D})$ (see Definition \ref{quadratic_representation}). Therefore
$$X=\frac{1}{2}\big(W(We_1e_1^*+e_1e_1^*W)+(We_1e_1^*+e_1e_1^*W) W\big)-e_1e_1^*.$$
Since $e_1e_1^*$ is a real matrix, it follows that $W(We_1e_1^*)=W^2e_1e_1^*=e_1e_1^*$ and $W(e_1e_1^*W)=(We_1e_1^*)W$. Therefore $X=We_1e_1^*W$ (which is equal to $W(e_1e_1^*W)$ and to $(We_1e_1^*)W$). Since $e_1$ is a real vector, it follows that $X=(We_1)(e_1^*W)=(We_1)(We_1)^*$, and $||We_1||^2=(e_1^*W)(We_1)=e_1^*W^2e_1=1$. Moreover, since $W$ is hermitian matrix, its diagonal is real, and in particular $e_1^*We_1\in \mathbb{R}$, i.e. the first component of $We_1$ is real.

Conversely, if $\mathbb{D}\ne \mathbb{O}$, then $M_n(\mathbb{D})$ is associative, therefore $(uu^*)^2=u(u^*u)u^*=uu^*$, i.e. $uu^*$ is idempotent. Moreover it is primitive, since $\mathrm{Tr}(uu^*)=||u||^2=1$.

Assume now that $\mathbb{D}=\mathbb{O}$ and $n=3$ and write $u=\left[
\begin{array}{c}a\\x\\y
\end{array}
\right]$ where $a\in \mathbb{R}$, $x,y\in \mathbb{O}$ and $a^2+|x|^2+|y|^2=1$. Then $uu^*=\left[
\begin{array}{ccc}a^2&a\overline{x}&a\overline{y}\\ax&|x|^2&x\overline{y}\\ay&y\overline{x}&|y|^2
\end{array}
\right]$, and using $a\in \mathbb{R}$, $a^2+|x|^2+|y|^2=1$ and Lemma \ref{Hurwitz algebra identities} a short calculation shows that $(uu^*)^2=uu^*$, i.e. $uu^*$ is idempotent. Moreover, it is primitive, since $\mathrm{Tr}(uu^*)=1$.
\end{proof}

\begin{rmk}
If $\mathbb{D}\ne \mathbb{O}$, then the associativity implies $ux(ux)^*=ux\overline{x}u^*=uu^*$ for each $u\in \mathbb{D}^n$ and each $x\in \mathbb{D}$ with $|x|=1$, so in this case $e_1^*u\in \mathbb{R}$ is not an additional assumption, since the vectors $ux$ and $u$ give the same matrix $uu^*$. However, this is not true if $\mathbb{D}=\mathbb{O}$. For example, if $\{1,f_2,f_3,\ldots ,f_8\}$ is the standard basis for $\mathbb{O}$, let $u=\frac{1}{2}\left[
\begin{array}{c}f_2\\f_3\\1+f_7
\end{array}
\right]$. Then $||u||=1$, but $uu^*=\frac{1}{4}\left[
\begin{array}{ccc}1&-f_4&f_2+f_8\\f_4&1&f_3+f_5\\-f_2-f_8&-f_3-f_5&2
\end{array}
\right]$ is not an idempotent. Moreover, using equations (5.7), (5.3) and (5.11) of \cite{SpV} it is possible to show that $(uu^*)^3-(uu^*)^2+\frac{1}{16}I=0$, i.e. the minimal polynomial of $uu^*$ is $t^3-t^2+\frac{1}{16}$. This polynomial has a negative root, therefore the spectral theorem implies that $uu^*$ is not even an element of $H_3(\mathbb{O})^+$.
\end{rmk}

To apply the previous lemma to Lemma \ref{primitive idempotents} it remains to prove the following technical result, which is clear in the real and complex case.

\begin{lemma}\label{scalar product with rank1}
Let $n$ be a positive integer and $\mathbb{D}\in \{\mathbb{R},\mathbb{C},\mathbb{H},\mathbb{O}\}$, where $\mathbb{D}=\mathbb{O}$ only if $n=3$. Then $\langle X,uu^*\rangle =\frac{1}{2}\big((u^*X)u+u^*(Xu)\big)=\mathrm{Re}(u^*Xu)$ for all $X\in H_n(\mathbb{D})$ and all $u\in \mathbb{D}^n$. Moreover, if $\mathbb{D}\ne \mathbb{O}$, then $\langle vv^*,uu^*\rangle =|v^*u|^2=|u^*v|^2$ for all $u,v\in \mathbb{D}^n$.
\end{lemma}
\begin{proof}
Write $u=\left[
\begin{array}{c}u_1\\\vdots\\u_n
\end{array}
\right]$ and $v=\left[
\begin{array}{c}v_1\\\vdots\\v_n
\end{array}
\right]$ with $u_l,v_l\in \mathbb{D}$ for $l=1,\ldots ,n$. Then $uu^*=\left[
\begin{array}{cccc}|u_1|^2&u_1\overline{u_2}&\cdots&u_1\overline{u_n}\\u_2 \overline{u_1}&|u_2|^2&\cdots&u_2\overline{u_n}\\ \vdots&\vdots&\ddots&\vdots\\u_n\overline{u_1}&u_n \overline{u_2}&\cdots&|u_n|^2
\end{array}
\right]$ and $vv^*=\left[
\begin{array}{cccc}|v_1|^2&v_1\overline{v_2}&\cdots&v_1\overline{v_n}\\v_2 \overline{v_1}&|v_2|^2&\cdots&v_2\overline{v_n}\\ \vdots&\vdots&\ddots&\vdots\\v_n\overline{v_1}&v_n \overline{v_2}&\cdots&|v_n|^2
\end{array}
\right]$, and using Lemma \ref{Hurwitz algebra identities} we obtain
$$\langle vv^*,uu^*\rangle =\frac{1}{2}\mathrm{Tr}\big((vv^*)(uu^*)+(uu^*)(vv^*)\big)=\frac{1}{2}\sum _{l=1}^n\sum _{m=1}^n\Big((v_l\overline{v_m})(u_m\overline{u_l})+(u_l\overline{u_m})(v_m\overline{v_l})\Big)$$
$$=\sum _{l,m=1}^n\langle u_m\overline{u_l},v_m\overline{v_l}\rangle _{\mathbb{D}}=\sum _{l,m=1}^n\langle u_m,(v_m\overline{v_l})u_l\rangle _{\mathbb{D}}=\frac{1}{2}\sum _{l,m=1}^n\Big((\overline{u_l}(v_l\overline{v_m}))u_m+\overline{u_m}((v_m\overline{v_l})u_l)\Big)$$
$$=\frac{1}{2}\Big((u^*(vv^*))u+u^*((vv^*)u)\Big).$$
If $\mathbb{D}\ne \mathbb{O}$, then the matrix algebras are associative and the above equality proves the second part of the lemma. Moreover, for each $\mathbb{D}\in \{\mathbb{R},\mathbb{C},\mathbb{H},\mathbb{O}\}$ the above equality proves the first part of the lemma for matrices $X$ of the form $vv^*$ with $v\in \mathbb{D}^n$. However, by the spectral theorem and Lemma \ref{primitive idempotents uu^*} each element of $H_n(\mathbb{D})$ is a real linear combination of such matrices, which completes the proof of the lemma.
\end{proof}

Lemma \ref{primitive idempotents} now implies the following two corollaries.

\begin{cor}\label{primitive idempotents2}
Let $n$ be a positive integer, $\mathbb{D}\in \{\mathbb{R},\mathbb{C},\mathbb{H}\}$ and let $A\colon H_n(\mathbb{D})\to H_n(\mathbb{D})$ be a linear map. Then the following holds.
\begin{enumerate}
\item[(a)]
The following are equivalent:
\begin{itemize}
\item
$A\in \mathrm{End}(H_n(\mathbb{D})^+)$.
\item
$A(uu^*)\in H_n(\mathbb{D})^+$ for all $u\in \mathbb{D}^n$.
\item
$v^*A(uu^*)v\ge 0$ for all $u,v\in \mathbb{D}^n$.
\end{itemize}
\item[(b)]
$A\in \mathfrak{s}(H_n(\mathbb{D})^+)$ if and only if $v^*A(uu^*)v\ge 0$ for all $u,v\in \mathbb{D}^n$ that satisfy $v^*u=0$.
\item[(c)]
$A\in \mathfrak{g}(H_n(\mathbb{D})^+)$ if and only if $v^*A(uu^*)v=0$ for all $u,v\in \mathbb{D}^n$ that satisfy $v^*u=0$.
\end{enumerate}
\end{cor}

\begin{cor}\label{primitive idempotents2O}
Let $A\colon H_3(\mathbb{O})\to H_3(\mathbb{O})$ be a linear map. Then the following holds.
\begin{enumerate}
\item[(a)]
The following are equivalent:
\begin{itemize}
\item
$A\in \mathrm{End}(H_3(\mathbb{O})^+)$.
\item
$A(uu^*)\in H_3(\mathbb{O})^+$ for all $u\in \mathbb{O}^3$ with real first component.
\item
$\mathrm{Re}(v^*A(uu^*)v)\ge 0$ for all $u,v\in \mathbb{O}^3$ with real first components.
\end{itemize}
\item[(b)]
$A\in \mathfrak{s}(H_3(\mathbb{O})^+)$ if and only if $\mathrm{Re}(v^*A(uu^*)v)\ge 0$ for all $u,v\in \mathbb{O}^3$ with real first components that satisfy $\mathrm{Re}(v^*(uu^*)v)=0$.
\item[(c)]
$A\in \mathfrak{g}(H_3(\mathbb{O})^+)$ if and only if $\mathrm{Re}(v^*A(uu^*)v)=0$ for all $u,v\in \mathbb{O}^3$ with real first components that satisfy $\mathrm{Re}(v^*(uu^*)v)=0$.
\end{enumerate}
\end{cor}

We finish this section with results describing behavior of a positive linear map $A$ that satisfies $\langle A(uu^*),vv^*\rangle =0$ for some $u,v\in \mathbb{D}^n$. These results will be very useful in Section 4, since there exist many vectors $u,v\in \mathbb{D}^n$ that satisfy $\langle B(uu^*),vv^*\rangle =0$ and $\langle uu^*,vv^*\rangle =0$, where $B$ is the map defined in Theorem \ref{B-counterexample}.

\begin{lemma}\label{zeros of positive map}
Let $n$ be a positive integer, $\mathbb{D}\in \{\mathbb{R},\mathbb{C},\mathbb{H},\mathbb{O}\}$ (with $\mathbb{D}=\mathbb{O}$ only if $n=3$), and let $A$ be an endomorphism of the cone $H_n(\mathbb{D})^+$. Assume that there exist vectors $u,v\in \mathbb{D}^n$ such that $uu^*$ and $vv^*$ are idempotents and $\langle A(uu^*),vv^*\rangle =0$. If $v'\in \mathbb{D}^n$ is any vector such that $(v+av')(v+av')^*$ is a multiple of an idempotent for each $a\in \mathbb{R}$, then $\mathrm{Re}(v'^*A(uu^*)v)=0$. Moreover, if $(v+v'x)(v+v'x)^*$ is a multiple of an idempotent for each $x\in \mathbb{D}$, then $v'^*(A(uu^*)v)=0$.
\end{lemma}
\begin{proof}
Since $(v+av')(v+av')^*$ is a multiple of an idempotent for each $a\in \mathbb{R}$ and $\langle A(uu^*),vv^*\rangle =0$, Lemmas \ref{primitive idempotents} and \ref{scalar product with rank1} imply that
$$0\le \langle A(uu^*),(v+av')(v+av')^*\rangle =\mathrm{Re}\Big((v+av')^*A(uu^*)(v+av')\Big)$$
$$=a^2\, \mathrm{Re}(v'^*A(uu^*)v')+2a\, \mathrm{Re}(v'^*A(uu^*)v)$$
for all $a\in \mathbb{R}$. However, this is clearly possible only if $\mathrm{Re}(v'^*A(uu^*)v)=0$, which proves the first part of the lemma.

For the second part let $\{f_1,\ldots ,f_d\}$ be the standard basis of $\mathbb{D}$ and write $v'^*(A(uu^*)v)=\sum _{m=1}^da_mf_m$ for some $a_1,\ldots ,a_d\in \mathbb{R}$. From the first part of the lemma and Corollary \ref{Hurwitz algebra real parts} we know that $\mathrm{Re}\Big((xv'^*)(A(uu^*)v)\Big)=\mathrm{Re} \Big(x\big(v'^*(A(uu^*)v)\big)\Big)=0$ for each $x\in \mathbb{D}$, in particular $\mathrm{Re}\Big(f_l\big(v'^*(A(uu^*)v)\big)\Big)=0$ for $l=1,\ldots ,d$. However, $\mathrm{Re}\Big(f_l\big(v'^*(A(uu^*)v)\big)\Big)= \left\{
\begin{array}{ll}a_l&\mathrm{if}\, l=1\\-a_l&\mathrm{if}\, l\ne 1
\end{array}
\right.$, therefore $v'^*(A(uu^*)v)=0$.
\end{proof}

\begin{rmk}
Note that $vv^*$ is idempotent for each normalized $v\in \mathbb{D}^n$ if $\mathbb{D}\in \{\mathbb{R},\mathbb{C},\mathbb{H}\}$, so the additional conditions on the matrices of such forms in the previous lemma are needed only if $\mathbb{D}=\mathbb{O}$.
\end{rmk}

\begin{cor}\label{v*uu*vO}
Let $u=\left[
\begin{array}{c}u_1\\u_2\\u_3
\end{array}
\right] ,v=\left[
\begin{array}{c}v_1\\v_2\\v_3
\end{array}
\right]\in \mathbb{O}^3$ be nonzero vectors with $u_1,v_1\in \mathbb{R}$. Then $\langle uu^*,vv^*\rangle =0$ if and only if any of the following relations hold:
\begin{enumerate}
\item
$u_2=u_3=v_1=0$.
\item
$u_3\ne 0$, $u_2=0$ and $v_3=-u_1v_1\overline{u_3}^{-1}$.
\item
$u_2\ne 0$ and $v_2=-u_1v_1\overline{u_2}^{-1}-(\overline{u_2}^{-1} \overline{u_3})v_3$.
\end{enumerate}
\end{cor}
\begin{proof}
First we compute
$$uu^*=\left[
\begin{array}{ccc}u_1^2&u_1\overline{u_2}&u_1\overline{u_3}\\ u_1u_2&|u_2|^2&u_2\overline{u_3}\\u_1u_3&u_3\overline{u_2}&|u_3|^2
\end{array}
\right]\quad \mathrm{and}\quad (uu^*)v=\left[
\begin{array}{c}u_1^2v_1+u_1\overline{u_2}v_2+u_1 \overline{u_3}v_3\\u_1v_1u_2+|u_2|^2v_2+(u_2\overline{u_3})v_3\\u_1v_1u_3+(u_3 \overline{u_2})v_2+|u_3|^2v_3
\end{array}
\right] .$$
Using Lemma \ref{Hurwitz algebra identities} a short computation shows that if $u$ and $v$ satisfy one of the first two conditions above, then $(uu^*)v=0$, and $\langle uu^*,vv^*\rangle =0$ by Lemma \ref{scalar product with rank1}. In the third case we get
$$(uu^*)v=\left[
\begin{array}{c}u_1\overline{u_3}v_3-u_1\overline{u_2}((\overline{u_2}^{-1}\overline{u_3})v_3)\\0\\0
\end{array}
\right],$$
and
$$\langle uu^*,vv^*\rangle =\mathrm{Re}(v^*(uu^*)v)=\mathrm{Re}\Big(u_1v_1\overline{u_3}v_3-u_1v_1\overline{u_2}((\overline{u_2}^{-1}\overline{u_3})v_3) \Big)=0$$
by Corollary \ref{Hurwitz algebra real parts} and Lemmas \ref{Hurwitz algebra identities} and \ref{scalar product with rank1}.

Conversely, suppose that $u$ and $v$ satisfy $\langle uu^*,vv^*\rangle =\mathrm{Re}(v^*(uu^*)v)=0$. If $u_2=u_3=0$, then $u_1\ne 0$ and $\mathrm{Re}(v^*(uu^*)v)=0$ implies $v_1=0$.

If $u_2=0$ and $u_3\ne 0$, then $(uu^*)v=\left[
\begin{array}{c}u_1^2v_1+u_1\overline{u_3}v_3\\0\\u_1v_1u_3+|u_3|^2v_3
\end{array}
\right]$. The matrix $(v+xe_3)(v+xe_3)^*$ is a multiple of a primitive idempotent for each $x\in \mathbb{O}$ by Lemma \ref{primitive idempotents uu^*}, and since the identity is an endomorphism of the cone $H_3(\mathbb{O})^+$, Lemma \ref{zeros of positive map} implies that $e_3^*(uu^*)v=0$. Therefore $u_1v_1u_3+|u_3|^2v_3=0$ and we get the condition 2 of the corollary.

It remains to consider the case $u_2\ne 0$. However, in this case similarly as above we get $e_2^*(uu^*)v=0$, which immediately gives us the condition 3.
\end{proof}

\begin{cor}\label{zeros of cross-positive map}
Let $n$ be a positive integer, $\mathbb{D}\in \{\mathbb{R},\mathbb{C},\mathbb{H},\mathbb{O}\}$ (with $\mathbb{D}=\mathbb{O}$ only if $n=3$), and let $A\colon H_n(\mathbb{D})\to H_n(\mathbb{D})$ be a linear map defined by $A(X)=A_1(X)+A_2(X)$ where $A_1\in \mathfrak{g}(H_n(\mathbb{D})^+)$ and $A_2\in \mathrm{End}(H_n(\mathbb{D})^+)$. Assume that there exist nonzero vectors $u,v\in \mathbb{D}^n$ such that $uu^*$ and $vv^*$ are idempotents, $\langle uu^*,vv^*\rangle =0$ and $\langle A(uu^*),vv^*\rangle =0$. If $(v+ux)(v+ux)^*$ is a multiple of an idempotent for each $x\in \mathbb{D}$, then $u^*(A_1(uu^*)v)=u^*(A(uu^*)v)$.

In particular, if $\mathbb{D}\in \{\mathbb{R},\mathbb{C},\mathbb{H}\}$ and $A(X)=HX+XH^*+A_2(X)$ for some $H\in M_n(\mathbb{D})$ and $A_2\in \mathrm{End}(H_n(\mathbb{D})^+)$, and there exist nonzero vectors $u,v\in \mathbb{D}^n$ such that $v^*u=0$ and $v^*A(uu^*)v=0$, then
$$v^*Hu=\frac{1}{||u||^2}v^*A(uu^*)u.$$
\end{cor}
\begin{proof}
Note that the second part of the corollary is just a special case of the first part, and the first part follows immediately from Lemma \ref{zeros of positive map} applied to the map $A_2=A-A_1$, because $\langle A_1(uu^*),vv^*\rangle =0$ by Lemma \ref{primitive idempotents}(c), and therefore $\langle A_2(uu^*),vv^*\rangle =0$.
\end{proof}

\section{Cross-positivity of $B$}

In this section we will show that the map $B$ defined in Theorem \ref{B-counterexample} is a generator of 1-parameter semigroup of positive maps on $H_n(\mathbb{D})$. To show this we will need the following two technical lemmas. The first one was proved in \cite[Lemma 2.2]{H}.

\begin{lemma}\label{inequalityHa}
Let $n\ge 2$ and let $a,c_1,\ldots ,c_n$ be positive real numbers. Then the inequality
$$\frac{1}{a+c_1\alpha _1}+\frac{1}{a+c_2\alpha _2}+\cdots +\frac{1}{a+c_n\alpha _n}\le 1$$
holds for all positive real numbers $\alpha _1,\alpha _2,\ldots ,\alpha _n$ that satisfy $\alpha _1\cdots \alpha _n=1$ if and only if $a\ge n-1$ and $(c_1\cdots c_n)^{\frac{1}{n}}\ge n-a$.
\end{lemma}

\begin{lemma}\label{determinant}
Let $n\ge 2$ and let $a_1,a_2,\ldots ,a_n,b$ be arbitrary real numbers. Then
$$\det\left[
\begin{array}{ccccc}a_1&b&b&\cdots&b\\b&a_2&b&\cdots&b\\b&b&a_3& \cdots&b\\\vdots&\vdots&\vdots&\ddots&\vdots\\b&b&b& \cdots&a_n
\end{array}
\right] =\prod _{l=1}^n(a_l-b)+b\sum _{l=1}^n\prod _{m\ne l}(a_m-b).$$
\end{lemma}
\begin{proof}
We will prove the lemma by induction on $n$. Denote the above determinant by $D_n$. For $n=2$ the equality
$$D_2=a_1a_2-b^2=(a_1-b)(a_2-b)+b(a_1-b+a_2-b)$$
clearly holds. Assume now that $n\ge 3$ and that the lemma holds for $n-1$. Subtracting the first column of the above matrix from the others we obtain
$$\det\left[
\begin{array}{ccccc}a_1&b&b&\cdots&b\\b&a_2&b&\cdots&b\\b&b&a_3& \cdots&b\\\vdots&\vdots&\vdots&\ddots&\vdots\\b&b&b& \cdots&a_n
\end{array}
\right] =\det\left[
\begin{array}{ccccc}a_1&b-a_1&b-a_1&\cdots&b-a_1 \\b&a_2-b&0&\cdots&0\\b&0&a_3-b&\cdots&0\\\vdots&\vdots&\vdots&\ddots&\vdots\\ b&0&0&\cdots&a_n-b
\end{array}\right] ,$$
and now by induction we get
$$D_n=(a_n-b)D_{n-1}-b(b-a_1)(a_2-b)\cdots(a_{n-1}-b)$$
$$=(a_n-b)\left(\prod _{l=1}^{n-1}(a_l-b)+b\sum _{l=1}^{n-1}\prod _{m\ne l,n}(a_m-b)\right)+b\prod _{l=1}^{n-1}(a_l-b)=\prod _{l=1}^n(a_l-b)+b\sum _{l=1}^n\prod _{m\ne l}(a_m-b).$$
\end{proof}

Now we can prove cross-positivity of the map $B$ defined in Theorem \ref{B-counterexample}. We will prove it separately for $\mathbb{D}\in \{\mathbb{R},\mathbb{C},\mathbb{H}\}$ and for $\mathbb{D}=\mathbb{O}$.

\begin{prop}\label{Bcross-positive}
Let $n\ge 3$ be a positive integer, $\mathbb{D}\in \{\mathbb{R},\mathbb{C},\mathbb{H}\}$ and let $B\colon H_n(\mathbb{D})\to H_n(\mathbb{D})$ be the linear map defined in Theorem \ref{B-counterexample}. Then $B$ is generator of 1-parameter semigroup of linear maps on $H_n(\mathbb{D})$ which leave the cone $H_n(\mathbb{D})^+$ invariant.
\end{prop}
\begin{proof}
By Corollary \ref{primitive idempotents2}(b) it suffices to prove that $v^*B(uu^*)v\ge 0$ for all $u,v\in \mathbb{D}^n$ which satisfy $v^*u=0$. Write $u=\left[
\begin{array}{c}u_1\\\vdots\\u_n
\end{array}
\right]$ and $v=\left[
\begin{array}{c}v_1\\\vdots\\v_n
\end{array}
\right]$. The condition $v^*u=0$ is then equivalent to $\sum _{l=1}^n\overline{v_l}u_l=0$. In the sequel it will be convenient to have cyclic indices, therefore we define $u_{n+1}=u_1$, $v_{n+1}=v_1$, $u_0=u_n$ and $v_0=v_n$.

Since
$$B(uu^*)=B\left(\left[
\begin{array}{cccc}|u_1|^2&u_1\overline{u_2}&\cdots&u_1\overline{u_n}\\ u_2\overline{u_1}&|u_2|^2&\cdots&u_2\overline{u_n}\\ \vdots&\vdots&\ddots&\vdots\\u_n\overline{u_1}& u_n\overline{u_2}&\cdots&|u_n|^2
\end{array}
\right]\right)$$
$$=\left[
\begin{array}{ccccc}|u_2|^2-|u_1|^2&-pu_1\overline{u_2}&-pu_1\overline{u_3}&\cdots&-pu_1\overline{u_n}\\-pu_2\overline{u_1}&(n-1)^2(|u_3|^2-|u_2|^2)&-qu_2\overline{u_3}&\cdots&-qu_2\overline{u_n}\\-pu_3\overline{u_1}&-qu_3\overline{u_2}&(n-1)^2(|u_4|^2-|u_3|^2)&\cdots&-qu_3\overline{u_n}\\\vdots&\vdots&\vdots&\ddots& \vdots\\-pu_n\overline{u_1}&-qu_n\overline{u_2}&-qu_n\overline{u_3}&\cdots&(n-1)^2(|u_1|^2-|u_n|^2)
\end{array}
\right] ,$$
$p=\frac{(n-2)(n^2-n-1)}{2}$ and $q=(n-1)(n^2-n-1)$, we get
$$v^*B(uu^*)v=(|u_2|^2-|u_1|^2)|v_1|^2-p\overline{v_1}u_1\sum _{m=2}^n\overline{u_m}v_m-p\sum _{m=2}^n\overline{v_m}u_m\overline{u_1}v_1-q\sum _{l,m=2}^n\overline{v_l}u_l\overline{u_m}v_m+$$
$$+(q-(n-1)^2)\sum _{l=2}^n|u_l|^2|v_l|^2+(n-1)^2\sum _{l=2}^n|u_{l+1}|^2|v_l|^2$$
$$=(|u_2|^2-|u_1|^2)|v_1|^2-(n-2)(n^2-n-1)\mathrm{Re}\Big(\sum _{m=2}^n\overline{v_m}u_m\overline{u_1}v_1\Big)-$$
$$-(n-1)(n^2-n-1)\sum _{l,m=2}^n\overline{v_l}u_l\overline{u_m}v_m+$$
$$+n(n-1)(n-2)\sum _{l=2}^n|u_l|^2|v_l|^2+(n-1)^2\sum _{l=2}^n|u_{l+1}|^2|v_l|^2$$
$$=|u_2|^2|v_1|^2-n(n-1)|u_1|^2|v_1|^2+n(n-1)(n-2)\sum _{l=2}^n|u_l|^2|v_l|^2+(n-1)^2\sum _{l=2}^n|u_{l+1}|^2|v_l|^2,$$
where we used the identity $\sum _{l=2}^n\overline{v_l}u_l=-\overline{v_1}u_1$.

If $u_1=0$, then it is clear that $v^*B(uu^*)v\ge 0$, therefore we will assume that $u_1\ne 0$. Then $v_1=-\sum _{l=2}^n\overline{u_1}^{-1}\overline{u_l}v_l$ and
$$v^*B(uu^*)v=\left(\frac{|u_2|^2}{|u_1|^2}-n(n-1)\right)\sum _{l,m=2}^n\overline{v_l}u_l\overline{u_m}v_m+$$
$$+n(n-1)(n-2) \sum _{l=2}^n|u_l|^2|v_l|^2+(n-1)^2\sum _{l=2}^n|u_{l+1}|^2|v_l|^2$$
$$=\left[
\begin{array}{ccc}\overline{v_2}&\cdots& \overline{v_n}
\end{array}
\right] X_u\left[
\begin{array}{c}v_2\\\vdots\\v_n
\end{array}
\right]$$
where
$$X_u=\left[
\begin{array}{cccc}(\frac{|u_2|^2}{|u_1|^2}+r)|u_2|^2+t|u_3|^2&(\frac{|u_2|^2}{|u_1|^2}-s)u_2\overline{u_3}&\cdots&(\frac{|u_2|^2}{|u_1|^2}-s)u_2\overline{u_n}\\(\frac{|u_2|^2}{|u_1|^2}-s)u_3\overline{u_2}&(\frac{|u_2|^2}{|u_1|^2}+r)|u_3|^2+t|u_4|^2&\cdots&(\frac{|u_2|^2}{|u_1|^2}-s)u_3\overline{u_n}\\\vdots&\vdots&\ddots&\vdots\\ (\frac{|u_2|^2}{|u_1|^2}-s)u_n\overline{u_2}&(\frac{|u_2|^2}{|u_1|^2}-s)u_n\overline{u_3}&\cdots&(\frac{|u_2|^2}{|u_1|^2}+r)|u_n|^2+t|u_1|^2
\end{array}
\right],$$
where we denoted $r=n(n-1)(n-3)$, $s=n(n-1)$ and $t=(n-1)^2$.

Using self-duality of the cone $H_{n-1}(\mathbb{D})^+$ we see that in order to prove $v^*B(uu^*)v\ge 0$ for all $u,v\in \mathbb{D}^n$ with $v^*u=0$ and $u_1\ne 0$ it suffices to prove that the matrix $X_u$ is positive semidefinite for all $u\in \mathbb{D}^n$ with $u_1\ne 0$. Moreover, the set $\{X_u;u_m\ne 0\, \mathrm{for}\, \mathrm{all}\, m=1,\ldots ,n\}$ is clearly dense in the set $\{X_u;u_1\ne 0\}$, and since the cone $H_{n-1}(\mathbb{D})^+$ is closed, it suffices to prove that $X_u\in H_{n-1}(\mathbb{D})^+$ for all $u\in \mathbb{D}^n$ with all components nonzero. However, in this case
$$X_u=\left[
\begin{array}{ccc}u_2\\&\ddots\\&&u_n
\end{array}
\right] Y_u\left[
\begin{array}{ccc}u_2\\&\ddots\\&&u_n
\end{array}
\right] ^*$$
where
$$Y_u=\left[
\begin{array}{cccc}\frac{|u_2|^2}{|u_1|^2}+r+t\frac{|u_3|^2}{|u_2|^2}&\frac{|u_2|^2}{|u_1|^2}-s&\cdots&\frac{|u_2|^2}{|u_1|^2}-s\\\frac{|u_2|^2}{|u_1|^2}-s&\frac{|u_2|^2}{|u_1|^2}+r+t\frac{|u_4|^2}{|u_3|^2}&\cdots&\frac{|u_2|^2}{|u_1|^2}-s\\\vdots&\vdots&\ddots&\vdots\\\frac{|u_2|^2}{|u_1|^2}-s&\frac{|u_2|^2}{|u_1|^2}-s&\cdots&\frac{|u_2|^2}{|u_1|^2}+r+t\frac{|u_1|^2}{|u_n|^2}
\end{array}
\right] .$$
Clearly $X_u\in H_{n-1}(\mathbb{D})^+$ if and only if $Y_u\in H_{n-1}(\mathbb{D})^+$. However, $Y_u$ is a {\em real} matrix, therefore to prove the proposition it suffices to prove that $Y_u\in H_{n-1}(\mathbb{R})^+$ for all $u\in \mathbb{D}^n$ with all components nonzero. In particular, it suffices to show that all main subdeterminants of $Y_u$ are nonnegative.

For $l=1,\ldots ,n-1$ let $Y_{u,l}$ be the submatrix of $Y_u$ obtained by deleting the $l$-th row and column. It is clear that
$$Y_{u,l}\ge \frac{|u_2|^2}{|u_1|^2}\left[
\begin{array}{cccc}1&1&\cdots&1\\1&1&\cdots&1\\\vdots&\vdots& \ddots&\vdots\\1&1&\cdots&1
\end{array}
\right] +n(n-1)\left[
\begin{array}{cccc}n-3&-1&\cdots&-1\\-1&n-3&\cdots&-1\\ \vdots&\vdots&\ddots&\vdots\\-1&-1&\cdots&n-3
\end{array}
\right]$$
for each $l=1,\ldots ,n-1$. The first matrix in the above decomposition in clearly positive semidefinite, and the second one equals to $(n-2)I-ee^*$, where $e\in \mathbb{R}^{n-2}$ is the vector with all components equal to 1. The eigenvalues of $(n-2)I-ee^*$ are $n-2$ (with eigenvectors orthogonal to $e$) and $n-2-e^*e=0$ (with eigenvector $e$), therefore $(n-2)I-ee^*$ is positive semidefinite matrix. Hence, $Y_{u,l}$ is positive semidefinite matrix, and in particular each its main subdeterminant is nonnegative. Therefore each main subdeterminant of $Y_u$ of size at most $n-2$ is nonnegative.

To finish the proof it remains to show that $\det Y_u\ge 0$ for each $u\in \mathbb{D}^n$ with all components nonzero. By Lemma \ref{determinant} we get
$$\det Y_u=\prod _{l=2}^n\left(r+s+t\frac{|u_{l+1}|^2}{|u_l|^2}\right)+\left(\frac{|u_2|^2}{|u_1|^2}-s\right)\sum _{l=2}^n\prod _{2\le m\le n,m\ne l}\left(r+s+t\frac{|u_{m+1}|^2}{|u_m|^2}\right)$$
$$=\prod _{l=2}^n\left(r+s+t\frac{|u_{l+1}|^2}{|u_l|^2}\right)\cdot \left(1+\Big(\frac{|u_2|^2}{|u_1|^2}-s\Big)\sum _{l=2}^n\frac{|u_l|^2}{(r+s)|u_l|^2+t|u_{l+1}|^2}\right)$$
$$=\prod _{l=2}^n\left(n(n-1)(n-2)+(n-1)^2\frac{|u_{l+1}|^2}{|u_l|^2}\right)\cdot$$
$$\cdot \left(1+\Big(\frac{|u_2|^2}{|u_1|^2}-n(n-1)\Big)\sum _{l=2}^n\frac{|u_l|^2}{n(n-1)(n-2)|u_l|^2+(n-1)^2|u_{l+1}|^2}\right),$$
therefore $\det Y_u\ge 0$ if and only if
\begin{equation}\label{ineq1}
1+\Big(\frac{|u_2|^2}{|u_1|^2}-n(n-1)\Big)\sum _{l=2}^n\frac{|u_l|^2}{n(n-1)(n-2)|u_l|^2+(n-1)^2|u_{l+1}|^2}\ge 0.
\end{equation}
Let $z=\left(\frac{|u_2|^2}{|u_1|^2}\right)^{\frac{1}{n-1}}$ and let $x_l=z\frac{|u_{l+1}|^2}{|u_l|^2}$ for $l=2,\ldots ,n$. Then
$$x_2x_3\cdots x_n=z^{n-1}\prod _{l=2}^n\frac{|u_{l+1}|^2}{|u_l|^2}=\frac{|u_2|^2}{|u_1|^2}\frac{|u_1|^2}{|u_2|^2}=1,$$
and the inequality (\ref{ineq1}) is equivalent to
$$1+\Big(z^{n-1}-n(n-1)\Big)\sum _{l=2}^n\frac{1}{n(n-1)(n-2)+\frac{(n-1)^2}{z}x_l}\ge 0,$$
or equivalently
\begin{equation}\label{ineq2}
\sum _{l=2}^n\frac{1}{\frac{n(n-1)(n-2)}{n(n-1)-z^{n-1}}+\frac{(n-1)^2}{z(n(n-1)-z^{n-1})}x_l}\le 1
\end{equation}
if $n(n-1)\ne z^{n-1}$. However, it is clear, that (\ref{ineq1}) holds if $n(n-1)\le z^{n-1}$, therefore we can assume that $n(n-1)>z^{n-1}$. To prove (\ref{ineq2}) we will use Lemma \ref{inequalityHa}. Clearly $\frac{n(n-1)(n-2)}{n(n-1)-z^{n-1}}>n-2$.
To apply Lemma \ref{inequalityHa} we have to show that
$$\frac{(n-1)^2}{z(n(n-1)-z^{n-1})}\ge n-1-\frac{n(n-1)(n-2)}{n(n-1)-z^{n-1}},$$
or equivalently $n-1\ge nz-z^n$. However, this inequality is equivalent to $(z-1)(z^{n-1}+\cdots +z+1-n)\ge 0$, which holds, since for $z>1$ both factors are positive and for $z<1$ both factors are negative. Therefore the inequality (\ref{ineq2}) holds by Lemma \ref{inequalityHa}, so $\det Y_u\ge 0$ for all $u\in \mathbb{D}^n$ with nonzero components, which completes the proof of the proposition.
\end{proof}

\begin{prop}\label{Bcross-positiveO}
The linear map $B\colon H_3(\mathbb{O})\to H_3(\mathbb{O})$ defined by
$$B\left(\left[
\begin{array}{ccc}a&x&\overline{y}\\\overline{x}&b&z\\ y&\overline{z}&c
\end{array}
\right]\right)=\left[
\begin{array}{ccc}b-a&-\frac{5}{2}x&-\frac{5}{2}\overline{y}\\-\frac{5}{2}\overline{x}&4(c-b)&-10z\\-\frac{5}{2}y&-10\overline{z}&4(a-c)
\end{array}
\right]$$ is generator of 1-parameter semigroup of linear maps on $H_3(\mathbb{O})$ which preserve $H_3(\mathbb{O})^+$.
\end{prop}
\begin{proof}
By Corollary \ref{primitive idempotents2O}(b) it suffices to prove that $\mathrm{Re}(v^*B(uu^*)v)\ge 0$ for all vectors $u,v\in \mathbb{O}^3$ with real first components which satisfy $\langle uu^*,vv^*\rangle =\mathrm{Re}(v^*(uu^*)v)=0$. Let $u=\left[
\begin{array}{c}u_1\\u_2\\u_3
\end{array}
\right]$ and $v=\left[
\begin{array}{c}v_1\\v_2\\v_3
\end{array}
\right]$ be such vectors. By a short computation we get
$$B(uu^*)=\left[
\begin{array}{ccc}|u_2|^2-u_1^2&-\frac{5}{2}u_1\overline{u_2}&-\frac{5}{2}u_1\overline{u_3}\\-\frac{5}{2}u_1u_2&4(|u_3|^2-|u_2|^2)&-10u_2\overline{u_3}\\-\frac{5}{2}u_1u_3&-10u_3\overline{u_2}&4(u_1^2-|u_3|^2)
\end{array}
\right],$$
$$B(uu^*)v=\left[
\begin{array}{c}(|u_2|^2-u_1^2)v_1-\frac{5}{2}u_1\overline{u_2}v_2-\frac{5}{2}u_1\overline{u_3}v_3\\-\frac{5}{2}u_1v_1u_2+4(|u_3|^2-|u_2|^2)v_2-10(u_2\overline{u_3})v_3\\-\frac{5}{2}u_1v_1u_3-10(u_3\overline{u_2})v_2+4(u_1^2-|u_3|^2)v_3
\end{array}
\right]$$
and
$$v^*(B(uu^*)v)=(|u_2|^2-u_1^2)v_1^2-5u_1v_1\, \mathrm{Re}(\overline{u_2}v_2+\overline{u_3}v_3)+4(|u_3|^2-|u_2|^2)|v_2|^2+4(u_1^2-|u_3|^2)|v_3|^2-$$
$$-10\big(\overline{v_2}((u_2\overline{u_3})v_3)+\overline{v_3}((u_3\overline{u_2})v_2)\big).$$

Now we consider all possible cases for $u$ and $v$, which were described in Corollary \ref{v*uu*vO}.

{\bf Case 1:} If $u_2=u_3=v_1=0$, then $B(uu^*)$ is a real matrix and $v^*B(uu^*)v=4u_1^2|v_3|^2\ge 0$.

{\bf Case 2:} If $u_3\ne 0$, $u_2=0$ and $v_3=-u_1v_1\overline{u_3}^{-1}$, then $v^*(B(uu^*)v)=4|u_3|^2|v_2|^2+4\frac{u_1^4v_1^2}{|u_3|^2}$, therefore $\mathrm{Re}(v^*B(uu^*)v)=4|u_3|^2|v_2|^2+4\frac{u_1^4v_1^2}{|u_3|^2}\ge 0$.

{\bf Case 3:} If $u_2\ne 0$ and $v_2=-u_1v_1\overline{u_2}^{-1}-(\overline{u_2}^{-1} \overline{u_3})v_3$, then by Lemma \ref{Hurwitz algebra identities} and Corollary \ref{Re(x^-1xyz)} we get
$$\mathrm{Re}(v^*B(uu^*)v)=|u_2|^2v_1^2+4\frac{u_1^2v_1^2|u_3|^2}{|u_2|^2}+4\frac{|u_3|^4|v_3|^2}{|u_2|^2}+12|u_3|^2|v_3|^2+4u_1^2|v_3|^2+$$
$$+5u_1v_1\, \mathrm{Re}\Big(\overline{u_2}((\overline{u_2}^{-1}\overline{u_3})v_3)-\overline{u_3}v_3\Big)+8u_1v_1(|u_3|^2-|u_2|^2)\mathrm{Re}\Big(u_2^{-1}((\overline{u_2}^{-1}\overline{u_3})v_3)\Big)+$$
$$+10u_1v_1\, \mathrm{Re}\Big(u_2^{-1}((u_2\overline{u_3})v_3)+\overline{v_3}u_3\Big)$$
$$=\Big(4\frac{u_1^2|u_3|^2}{|u_2|^2}+|u_2|^2\Big)v_1^2+4u_1v_1\Big(2\frac{|u_3|^2}{|u_2|^2}+3\Big)\mathrm{Re}(\overline{u_3}v_3)+4\Big(u_1^2+\frac{|u_3|^4}{|u_2|^2}+3|u_3|^2\Big)|v_3|^2.$$

This is a quadratic polynomial in $v_1$ with positive leading term and the discriminant
$$16\Big(\big(2\frac{|u_3|^2}{|u_2|^2}+3\big)^2u_1^2\big(\mathrm{Re}(\overline{u_3}v_3)\big)^2-\big(4\frac{u_1^2|u_3|^2}{|u_2|^2}+|u_2|^2\big)\big(u_1^2+\frac{|u_3|^4}{|u_2|^2}+3|u_3|^2\big)|v_3|^2\Big)$$
$$\le 16\Big(\big(2\frac{|u_3|^2}{|u_2|^2}+3\big)^2u_1^2|u_3|^2|v_3|^2-\big(4\frac{u_1^2|u_3|^2}{|u_2|^2}+|u_2|^2\big)\big(u_1^2+\frac{|u_3|^4}{|u_2|^2}+3|u_3|^2\big)|v_3|^2\Big)$$
$$=-16|v_3|^2\Big(4\frac{u_1^4|u_3|^2}{|u_2|^2}+3|u_2|^2|u_3|^2+u_1^2|u_2|^2+|u_3|^4-9u_1^2|u_3|^2 \Big)$$
$$=-16|v_3|^2\Big(3\big(\frac{u_1^4|u_3|^2}{|u_2|^2}+|u_2|^2|u_3|^2-2u_1^2|u_3|^2\big)+\big(\frac{u_1^4|u_3|^2}{|u_2|^2}+u_1^2|u_2|^2+|u_3|^4-3u_1^2|u_3|^2\big) \Big)\le 0,$$
where we used the inequality between arithmetic and geometric mean. So $\mathrm{Re}(v^*B(uu^*)v)\ge 0$, which completes the proof of the proposition.
\end{proof}

\section{Indecomposability of $B$}

In this section we show that the map $B$ defined in Theorem \ref{B-counterexample} cannot be written as a sum of an endomorphism of $H_n(\mathbb{D})^+$ and an element of the Lie algebra $\mathfrak{g}(H_n(\mathbb{D})^+)$. Since there exist many vectors $u,v\in \mathbb{D}^n$ satisfying $\langle uu^*,vv^*\rangle =0$ and $\langle B(uu^*),vv^*\rangle =0$, Corollary \ref{zeros of cross-positive map} will play an important role in the proof of this result.

\begin{prop}\label{nonexsistenceH}
Let $n\ge 3$ be a positive integer, $\mathbb{D}\in \{\mathbb{R},\mathbb{C},\mathbb{H},\mathbb{O}\}$ (with $\mathbb{D}=\mathbb{O}$ only if $n=3$), and let $B\colon H_n(\mathbb{D})\to H_n(\mathbb{D})$ be the linear map defined in Theorem \ref{B-counterexample}. Then $B$ cannot be written as a sum of an endomorphism of $H_n(\mathbb{D})^+$ and an element of the Lie algebra $\mathfrak{g}(H_n(\mathbb{D})^+)$.
\end{prop}
\begin{proof}
We prove the proposition by contradiction. Assume that $B\in \mathrm{End}(H_n(\mathbb{D})^+)+\mathfrak{g}(H_n(\mathbb{D})^+)$. Then by Corollaries \ref{Lie algebraRCH} and \ref{Lie algebraO} and Remark \ref{derivationsD->derivationsH_n(D)} there exist $B'\in \mathrm{End}(H_n(\mathbb{D})^+)$, $H\in M_n(\mathbb{D})$ with $\mathrm{Tr}(H)\in \mathbb{R}$ and a derivation $D$ on $\mathbb{D}$ such that $B(X)=HX+XH^*+A_D(X)+B'(X)$ for each $X\in M_n(\mathbb{D})$, where $A_D\colon H_n(\mathbb{D})\to H_n(\mathbb{D})$ is the derivation defined in Lemma \ref{derivation octonions->H_3(O)} (and in Remark \ref{derivationsD->derivationsH_n(D)}). First we will use Corollary \ref{zeros of cross-positive map} to determine the form of $H$.

For an arbitrary $l\in \{1,\ldots ,n\}$ let $u=e_l$ and let $v=\left[
\begin{array}{c}v_1\\\vdots\\v_n
\end{array}
\right] \in \mathbb{R}^n$ be an arbitrary vector with $v_{l-1}=v_l=0$ (where we used the usual identification $v_0=v_n$). It is clear that $v^*u=0$ and $v^*B(uu^*)v=0$. Since the vectors $u$ and $v$ are real, an easy computation shows that $(v+ux)(v+ux)^*$ is a multiple of an idempotent for each $x\in \mathbb{D}$. The matrix $uu^*$ is real, therefore $A_D(uu^*)=0$ and Corollary \ref{zeros of cross-positive map} implies that $v^*He_l=v^*B(e_le_l^*)e_l=0$, i.e. in the $l$-th column of $H$ only the $l$-th and the $(l-1)$-th entry can be nonzero. Since $l$ was arbitrary, $H$ has to be of the form $H=\left[
\begin{array}{cccc}h_{11}&h_{12}\\&h_{22}&\ddots\\&&\ddots& h_{n-1,n}\\h_{n1}&&&h_{nn}
\end{array}
\right]$ for some $h_{ij}\in \mathbb{D}$. When necessary we will use cyclic indices, i.e. we will identify $h_{n,n+1}=h_{n1}=h_{01}$.

For an arbitrary set $S\subseteq \{2,\ldots ,n\}$ and an arbitrary $x\in \mathbb{D}$ with $|x|=1$ let $u=\left[
\begin{array}{c}u_1\\\vdots\\u_n
\end{array}
\right] \in \mathbb{D}^n$ be the vector defined by
\begin{equation}\label{uSx}
u_l=\left\{
\begin{array}{cc}1&\mathrm{if}\, l\not \in S\\x&\mathrm{if}\, l\in S
\end{array}
\right.
\end{equation}
and let
\begin{equation}\label{vSx}
v=u-ne_1.
\end{equation}
Then $uu^*$ and $vv^*$ are multiples of idempotents by Lemma \ref{primitive idempotents uu^*}, and an easy computation shows that $\langle uu^*,vv^*\rangle =0$. Clearly the values of $v^*(B(uu^*)v)$ and $u^*(B(uu^*)v)$ are independent of $S$ and of $x$, and it is easy to compute that $v^*(B(uu^*)v)=0$ and $u^*(B(uu^*)v)=-\frac{n(n-1)(n-2)(n^2-n-1)}{2}$. Now let $y\in \mathbb{D}$ be arbitrary and observe that $u=w+zx$ for some vectors $w,z\in \mathbb{D}^n$ whose components are equal to either 0 or 1, and $w^*z=0$, $e_1^*w=1$ and $e_1^*z=0$. Using $|x|=1$ and Lemma \ref{Hurwitz algebra identities} we get
$$(e_1+uy)(e_1+uy)^*=(e_1+yw+xyz)(e_1^*+\overline{y}w^*+\overline{y}\cdot \overline{x}z^*)$$
$$=e_1e_1^*+e_1(\overline{y}w^*+\overline{y}\cdot \overline{x}z^*)+(wy+zxy)e_1^*+|y|^2ww^*+|y|^2(\overline{x}wz^*+xzw^*)+|y|^2zz^*$$
and after some computation, using $w^*z=0$, $e_1^*w=1$, $e_1^*z=0$, $||u||^2=w^*w+z^*z=n$, $|x|=1$, Lemma \ref{Hurwitz algebra identities} and $\mathrm{Re}((xy)\overline{x})=\mathrm{Re}(y)$, which follows from Corollary \ref{Hurwitz algebra real parts}, we get
$$\Big((e_1+uy)(e_1+uy)^*\Big)^2=\Big(1+2\, \mathrm{Re}(y)+n|y|^2\Big)\Big((e_1+uy)(e_1+uy)^*\Big),$$
i.e. $(e_1+uy)(e_1+uy)^*$ is a multiple of an idempotent for each $y\in \mathbb{D}$. Since $v=u-ne_1$, it follows that $(v+uy)(v+uy)^*$ is also a multiple of an idempotent for each $y\in \mathbb{D}$. Corollary \ref{zeros of cross-positive map} therefore implies that
\begin{equation}\label{zeros uvSx}
u^*\Big(\big(H(uu^*)+(uu^*)H^*+A_D(uu^*)\big)v \Big)=u^*(B(uu^*)v)=-\frac{n(n-1)(n-2)(n^2-n-1)}{2}
\end{equation}
for each $S\subseteq \{2,\ldots ,n\}$ and each $x\in \mathbb{D}$ with $|x|=1$, where $u$ and $v$ are defined by (\ref{uSx}) and (\ref{vSx}).

First let $x=-1$. Then $uu^*$ is a real matrix, therefore $A_D(uu^*)=0$ and (\ref{zeros uvSx}) implies that $v^*Hu=-\frac{(n-1)(n-2)(n^2-n-1)}{2}$.
For $S=\varnothing$ we get
$$(1-n)(h_{11}+h_{12})+\sum _{l=2}^n(h_{ll}+h_{l,l+1})=-\frac{(n-1)(n-2)(n^2-n-1)}{2},$$
for $S=\{2\}$ we get
$$(1-n)h_{11}+(n-1)h_{12}+\sum _{l=2}^n(h_{ll}+h_{l,l+1})-2h_{23}=-\frac{(n-1)(n-2)(n^2-n-1)}{2},$$
for $S=\{l\}$, where $l\in \{3,\ldots ,n\}$ we get
$$(1-n)(h_{11}+h_{12})+\sum _{l=2}^n(h_{ll}+h_{l,l+1})-2(h_{l-1,l}+h_{l,l+1})= -\frac{(n-1)(n-2)(n^2-n-1)}{2},$$
and for $S=\{2,3\}$ we get
$$(1-n)h_{11}+(n-1)h_{12}+\sum _{l=2}^n(h_{ll}+h_{l,l+1})-2h_{34}= -\frac{(n-1)(n-2)(n^2-n-1)}{2}.$$
The above equations immediately imply that $h_{12}=h_{23}=\cdots =h_{n-1,n}=h_{n1}=0$ and $(1-n)h_{11}+\sum _{l=2}^nh_{ll}=-\frac{(n-1)(n-2)(n^2-n-1)}{2}$. Since we showed that $H$ is a diagonal matrix, we will for simplicity write $h_l=h_{ll}$ for each $l=1,\ldots ,n$. Then
\begin{equation}\label{h1}
h_1=\frac{1}{n-1}\sum _{l=2}^nh_l+\frac{(n-2)(n^2-n-1)}{2}.
\end{equation}

Next we will show that $H\in M_n(Z(\mathbb{D}))$ where $Z(\mathbb{D})$ denotes the center of the algebra $\mathbb{D}$. To show this we will use the equality (\ref{zeros uvSx}) for the vectors $u$ and $v$ defined by (\ref{uSx}) and (\ref{vSx}) where $x\in \mathbb{D}$ is arbitrary with $|x|=1$ and $S=\{m\}$ for an arbitrary $m\in\{2,\ldots ,n\}$. For such $u$ and $v$ after a short computation using (\ref{h1}) and Lemma \ref{Hurwitz algebra identities} we get
$$u^*\Big(\big(H(uu^*)+(uu^*)H^*\big)v\Big)=n\Big((1-n)\overline{h_1}+\sum _{l=2}^n\overline{h_l}\Big)+(n-1)(\overline{x}\overline{h_m})x+ \overline{x}(\overline{h_m}x)-n\overline{h_m}$$
$$=-\frac{n(n-1)(n-2)(n^2-n-1)}{2}+(n-1)(\overline{x}\overline{h_m})x+ \overline{x}(\overline{h_m}x)-n\overline{h_m}$$
and
$$u^*(A_D(uu^*)v)=(n-1)D(\overline{x})x-\overline{x} D(x)=n\overline{x}D(\overline{x}),$$
where we used equalities $D(\overline{x})=-D(x)$ and $\overline{x}D(\overline{x})=D(\overline{x})x$ which were proved in Lemma \ref{derivation octonions->H_3(O)}. From (\ref{zeros uvSx}) now we get
$$(n-1)(\overline{x}\overline{h_m})x+ \overline{x}(\overline{h_m}x)-n\overline{h_m}+n\overline{x}D(\overline{x})=0,$$
therefore
$$D(x)=-D(\overline{x})=\frac{n-1}{n}x((\overline{x}\overline{h_m})x)+\frac{1}{n}\overline{h_m}x-x\overline{h_m}=\overline{h_m}x-x\overline{h_m},$$
where we used Lemma \ref{Hurwitz algebra identities} and the Moufang identity (i.e. (1.13)  of \cite{SpV}). Since $x\in \mathbb{D}$ with $|x|=1$ was arbitrary and $D$ is linear, we get $D(x)=\overline{h_m}x-x\overline{h_m}$ for all $x\in \mathbb{D}$ and all $m=2,\ldots ,n$. In particular, $\overline{h_2}x-x\overline{h_2}=\overline{h_m}x-x\overline{h_m}$ for each $x\in \mathbb{D}$ and each $m=3,\ldots ,n$, which is equivalent to $h_m-h_2\in Z(\mathbb{D})$ for $m=3,\ldots ,n$. On the other hand, the condition $\mathrm{Tr}(H)\in \mathbb{R}\subseteq Z(\mathbb{D})$ and the equation (\ref{h1}) yield $\sum _{l=2}^nh_l\in Z(\mathbb{D})$, therefore $h_1,h_2,\ldots ,h_n$ all belong to the center of $\mathbb{D}$, i.e. $H\in M_n(Z(\mathbb{D}))$. In particular, since $Z(\mathbb{H})=\mathbb{R}$ and $Z(\mathbb{O})=\mathbb{R}$, the matrix $H$ is always {\em complex}. It follows also that $D=0$.

As in the proof of Proposition \ref{Bcross-positive} let $t=(n-1)^2$, and let $u=\left[
\begin{array}{c}u_1\\\vdots\\u_n
\end{array}
\right]$ be an arbitrary {\em real} vector (and the indices are assumed to be cyclic). Then
$$B'(uu^*)=B(uu^*)-Huu^*-uu^*H^*$$
$$=\left[
\begin{array}{cccc}u_2^2-(1+2\mathrm{Re}(h_1))u_1^2&-(p+h_1+\overline{h_2})u_1u_2& \cdots&-(p+h_1+\overline{h_n})u_1u_n\\ -(p+h_2+\overline{h_1})u_1u_2&tu_3^2-(t+2 \mathrm{Re}(h_2))u_2^2&\cdots&-(q+h_2+ \overline{h_n})u_2u_n\\\vdots&\vdots&\ddots&\vdots \\-(p+h_n+\overline{h_1})u_1u_n&-(q+h_n+ \overline{h_2})u_2u_n&\cdots&tu_1^2-(t+2 \mathrm{Re}(h_n))u_n^2
\end{array}
\right]$$
is positive semidefinite matrix. Since $H$ is a complex matrix, the matrix $B'(uu^*)$ is also {\em complex}. Since it is positive semidefinite, each its main subdeterminant is nonnegative.

For $l=2,\ldots ,n$ let $X_{u,l}$ be the submatrix of $B'(uu^*)$ obtained by deleting the $l$-th row and column. Since the matrix $B'(uu^*)$ is positive semidefinite for each $u\in \mathbb{R}^n$, the determinant $\det X_{u,l}$ is nonnegative for each $u\in \mathbb{R}^n$ and each $l=2,\ldots ,n$. The determinant $\det X_{u,l}$ is a homogeneous polynomial of total degree $2(n-1)$, and it is of degree at most 4 in each of the variables $u_1,\ldots ,u_n$ and at most quadratic in $u_{l+1}$. Moreover, there exists no index $m_0\in \{1,\ldots ,n\}$ such that some monomial of $\det X_{u,l}$ is divisible by $\Big(\prod _{m=1}^{m_0}u_{l+m}^2\Big)u_{l+m_0+1}^4$. Since $\det X_{u,l}\ge 0$ for all $u\in \mathbb{R}^n$, it follows that the coefficient at the monomial $u_{l+1}^2u_{l+2}^2\cdots u_{l-1}^2$ in the determinant $\det X_{u,l}$ has to be nonnegative. This coefficient is equal to $\det Y_l$ where
$$Y_l=\left[
\begin{array}{cccccc}-1-2\mathrm{Re}(h_1)&\cdots&-p-h_1-\overline{h_{l-1}}&-p-h_1-\overline{h_{l+1}}&\cdots&-p-h_1-\overline{h_n}\\\vdots&\ddots&\vdots&\vdots&& \vdots\\-p-h_{l-1}-\overline{h_1}&\cdots&-t-2\mathrm{Re}(h_{l-1})&-q-h_{l-1}-\overline{h_{l+1}}&\cdots&-q-h_{l-1}-\overline{h_n}\\-p-h_{l+1}-\overline{h_1}&\cdots&-q-h_{l+1}-\overline{h_{l-1}}&-t-2\mathrm{Re}(h_{l+1})&\cdots&-q-h_{l+1}-\overline{h_n}\\\vdots&&\vdots&\vdots&\ddots & \vdots\\-p-h_n-\overline{h_1}&\cdots&-q-h_n-\overline{h_{l-1}}&-q-h_n-\overline{h_{l+1}}&\cdots&-t-2\mathrm{Re}(h_n)
\end{array}
\right] .$$
Let
$$R=\left[
\begin{array}{ccccc}1\\-1&1\\\vdots&&\ddots\\-1&&&1\\2-n&1&\cdots&1&1
\end{array}
\right].$$
Then $\det R=1$, therefore $\det Y_l=\det (RY_lR^T)$. However, using (\ref{h1}) and the definitions of $p$, $q$ and $t$ we can see that $RY_lR^T$ is a matrix of the form
$$\left[
\begin{array}{ccccc}*&*&\cdots&*&\overline{h_l}-\overline{h_1}+ \frac{(n-2)(n^2-n+1)}{2}\\ *&n(n-1)(n-3)&\cdots&-n(n-1)&0\\\vdots&\vdots&\ddots&\vdots&\vdots\\ *&-n(n-1)&\cdots&n(n-1)(n-3)&0\\h_l-h_1+\frac{(n-2)(n^2-n+1)}{2}&0&\cdots&0&0
\end{array}
\right]$$
and by Lemma \ref{determinant} its determinant is equal to
$$-n^{n-3}(n-1)^{n-3}(n-2)^{n-4}\left|h_l-h_1+\frac{(n-2)(n^2-n+1)}{2}\right|^2.$$
Since $\det Y_l\ge 0$, it follows that $h_l=h_1-\frac{(n-2)(n^2-n+1)}{2}$ for each $l=2,\ldots ,n$. However, (\ref{h1}) then implies $0=2-n$, which is a contradiction. Therefore $B\not \in \mathrm{End}(H_n(\mathbb{D})^+)+\mathfrak{g}(H_n(\mathbb{D})^+)$.
\end{proof}

\end{document}